[1994/12/01]
\documentclass[manuscript, reqno]{aomart}

\chardef\bslash=`\\ 

\newtheorem[{}\it]{thm}{Theorem}[section]

\newtheorem{lem}[thm]{Lemma}
\newtheorem{prop}[thm]{Proposition}

\theoremstyle{remark}

\usepackage{dsfont}
\usepackage{amssymb}

\usepackage{hyperref} 
\usepackage[numeric]{amsrefs}

\usepackage[english]{babel}

\newcommand{\symnn}{\mathrm{Sym}_{0}(3)}
\newcommand{\trr}{\mathrm{tr}}
\newcommand{\hkern}{e^{(t-s)\Delta}}

\theoremstyle{definition}
\newtheorem{defn}{Definition}[section]
\newtheorem{rem}{Remark}[section]
\newtheorem*[{}\it]{notation}{Notation}

\newtheorem*[{}\it]{rest}{\textsc{Theorem}}
\newtheorem*[{}\it]{proofoflemma}{Proof of Lemma}

\usepackage[textwidth=6in,textheight=9.5in, paperwidth=7.5in,paperheight=11.25in, inner=2.5cm, outer=2cm]{geometry} 

\usepackage{xcolor}

\usepackage{multicol}

\title[]{Dynamic Statistical Scaling in the Landau-de Gennes Theory of Nematic Liquid Crystals}

\author[]{Eduard Kirr, Mark Wilkinson and Arghir Zarnescu}

\subject{primary}{matsc2000}{123456789}
\subject{secondary}{matsc2000}{FFFF}

\volumenumber{1}
\issuenumber{1}
\publicationyear{2011}
\papernumber{1}
\startpage{1}
\endpage{}

\version{1.1}

\definecolor{orange}{rgb}{0.2,0.7,0.2}

\hypersetup{citecolor=red}

\usepackage{marvosym}

\begin{document}

\begin{abstract} In this article, we investigate the long time behaviour of a correlation function $c_{\mu_{0}}$ which is associated with a nematic liquid crystal system that is undergoing an isotropic-nematic phase transition. Within the setting of Landau-de Gennes theory, we confirm a hypothesis in the condensed matter physics literature on the average self-similar behaviour of this correlation function in the asymptotic regime at time infinity, namely
\begin{equation*}
\left\|c_{\mu_{0}}(r, t)-e^{-\frac{|r|^{2}}{8t}}\right\|_{L^{\infty}(\mathbb{R}^{3}, \,dr)}=\mathcal{O}(t^{-\frac{1}{2}}) \quad \text{as} \quad t\longrightarrow \infty.
\end{equation*}
In the final sections, we also pass comment on another scaling regime of the correlation function.
\end{abstract}
\maketitle


\section{Introduction}\label{introduckshun}

Nematic liquid crystals form a class of condensed matter systems whose constituent rod-like molecules give rise to rich macroscopic nonlinear phenomena. Thermotropic nematic liquid crystals are a well-studied subclass of nematics whose optical properties change dramatically with variation of system temperature. It is observed in the laboratory that above a certain temperature threshold, depending on the particular material under study, the rod-like molecules exhibit no local orientation preference. This is commonly known as the {\em isotropic} phase of the material. However, reducing the temperature of the material below this threshold results in the molecules arranging themselves along locally preferred directions, yielding the {\em nematic} phase. 

Models for the dynamics of this transition from isotropy to a nematic phase in liquid crystals present many challenging mathematical questions, both in continuum PDE theories and also in statistical mechanics. In the laboratory, the transition between the isotropic and nematic phases is observed to be effected by the seeding and subsequent growth of nematic `islands' in the ambient isotropic phase. The characteristic length scale $L\equiv L(t)$ of these island-like structures increases with time once the system is coerced into a transition of phase. It is argued in the physics literature (see \textsc{Bray} \cite{bray1}) that the domain coarsening of the nematic phase is a scaling phenomenon, namely the \emph{structure} of the nematic profile at late times looks \emph{statistically self-similar} to those profiles at an earlier time. We invite the reader to consult the extensive review article of \textsc{Bray} \cite{bray1} for a helpful introduction to this area of condensed matter physics. Laboratory experiments performed by \textsc{Pargellis et al.} \cite{PhysRevE.49.4250}, in which a nematic system was contrived to resemble a dynamic XY-model, support this argument.

It is the aim of this paper to study a model which captures such behaviour of thermotropic nematic systems in a mathematically rigorous manner. In particular, we wish to show that the average long time behaviour of solutions of a suitable model equation is spatially self similar, in an average sense. We work within the framework of continuum PDE theory, as opposed to the setting of statistical mechanics. 

As the isotropic-nematic phase transition is inherently dynamic, we require an appropriate evolution equation to model this phenomenon. The model we employ in this paper is the $L^{2}(\mathbb{R}^{3})$-gradient flow of the well-studied Landau-de Gennes energy. This energy is given by
\begin{equation}\label{ldg}
E_{\mathrm{LdG}}[Q]:=\int_{\mathbb{R}^{3}}\left(\frac{1}{2}|\nabla Q|^{2}+\frac{a(\vartheta)}{2}\mathrm{tr}\left(Q^{2}\right)-\frac{b}{3}\mathrm{tr}\left(Q^{3}\right)+\frac{c}{4}\mathrm{tr}\left(Q^{2}\right)Q\right)\,dx
\end{equation}
for appropriately regular maps $Q$, the gradient flow of which is
\begin{equation}\label{gradientflowings}
\frac{\partial Q}{\partial t}=\Delta Q-a(\vartheta)\,Q+b\left(Q^{2}-\frac{1}{3}\mathrm{tr}\left(Q^{2}\right)I\right)-c\,\mathrm{tr}\left(Q^{2}\right)Q.
\end{equation}
We refer the reader to \textsc{Ball} \cite{bally, bally2} or \textsc{Zarnescu} \cite{MR2962825} for an introduction to analytical and topological aspects of Landau-de Gennes theory.

In order to achieve our goal, we must first state in precise quantitative terms what we mean by \emph{structure} of nematic profiles and also by \emph{statistical self-similarity}. However, antecedent to both of these ideas is the notion of {\em order parameter}, which we now discuss in section \ref{orderparam} before introducing correlation functions in section \ref{structure} (related to {\em structure}), and also statistical solutions of PDEs in section \ref{statisticalsimilarity} (related to {\em statistical self-similarity}). Indeed, both of these concepts lead us naturally to the mathematical problem of obtaining a phase portrait for a quantity (a correlation function) describing the average behaviour of solutions of the gradient flow of \eqref{ldg} above. 

Although this article addresses a physical problem in the theory of phase transitions, the mathematical framework in which we address the problem itself is of wider import. Indeed, we argue that one might view the study of the correlation function as a convenient simplification of the infinite-dimensional dynamics generated by the PDE \eqref{gradientflowings} on a natural phase space, which captures one essential qualitative property of its solutions. 

\subsection{The de Gennes Q-tensor Order Parameter}\label{orderparam}
The first task when building a continuum model of such phase transition phenomena is to decide upon an appropriate order parameter that captures small scale material structure and allows one to distinguish between different phases of the material under study. In this article, we work with the Q-tensor order parameter which is able to describe both uniaxial and biaxial phases of nematic liquid crystals, as opposed to the director-field formalism of Ericksen-Leslie theory \cite{les1} or Oseen-Frank theory \cite{frank1}.

Let us suppose that to each point $x$ in a material domain $\Omega\subseteq\mathbb{R}^{3}$ we associate a probability density function $\rho_{x}$ on molecular orientations which lie in $\mathbb{S}^{2}$. In order to model the $\mathbb{Z}_{2}$ head-to-tail symmetry of nematic molecules, each density is endowed with the antipodal symmetry $\rho_{x}(n)=\rho_{x}(-n)$ for all $n\in\mathbb{S}^{2}$.  It was the idea of the physicist Pierre-Gilles de Gennes that one might consider the essential macroscopic information of the system to be contained in the matrix of second moments of $\rho_{x}$ with respect to molecular orientations. The \emph{de Gennes Q-tensor order parameter} (with respect to the density $\rho_{x}$) is defined to be
\begin{equation}\label{qtendef}
Q(x):=\int_{\mathbb{S}^{2}}\left(n\otimes n - \frac{1}{3}I\right)\,\rho_{x}(n)\,dn.
\end{equation}
It is a normalised matrix of second moments of $\rho_{x}$, and one may quickly check that $Q(x)$ is a traceless and symmetric $3\times 3$ matrix.

At the level of probability measures on molecular orientations, the uniform density $\overline{\rho}=1/|\mathbb{S}^{2}|$ corresponds to the isotropic phase of nematics. It is important to note that the term $-1/3 I$ (which contains no information about the system) is included in the above definition \eqref{qtendef} of the Q-tensor by convention, so as to render $Q$ identically zero in the isotropic phase. The Q-tensor may then be thought of as a macroscopic order parameter that measures the deviation of the system from isotropy. For a more detailed introduction to the Q-tensor order parameter, one might wish to consult \textsc{de Gennes and Prost} \cite{degennes1}, \textsc{Majumdar} \cite{majumdar1} or \textsc{Newton and Mottram} \cite{mottram1}. 

We denote the order parameter manifold of all such matrices by $\mathrm{Sym}_{0}(3)$, namely
\begin{equation}
\mathrm{Sym}_{0}(3):=\left\{Q\in\mathbb{R}^{3\times 3}\,:\, Q^{T}=Q\quad \text{and}\quad \mathrm{tr}(Q)=0\right\},\nonumber
\end{equation}
where $\mathrm{tr}:\mathbb{R}^{3\times 3}\rightarrow \mathbb{R}$ denotes the familar matrix trace operator. We subsequently refer to maps which take their values in $\textrm{Sym}_{0}(3)$ as nematic profiles.
\subsection{`Structure' of Nematic Profiles: Correlation Functions}\label{structure}
The measure of structure of nematic profiles $Q: \mathbb{R}^{3}\times (0, \infty)\rightarrow\mathrm{Sym}_{0}(3)$ we employ in this paper is related to the normalised two-point spatial correlation function $c:\mathbb{R}^{3}\times (0, \infty)\rightarrow\mathbb{R}$ defined by
\begin{equation}\label{corryfunny}
c(r, t):=\frac{\displaystyle \int_{\mathbb{R}^{3}}\mathrm{tr}\left(Q(x+r, t)Q(x, t)\right)\,dx}{\displaystyle \int_{\mathbb{R}^{3}}\mathrm{tr}\left(Q(x, t)^{2}\right)\,dx}.
\end{equation}
One may think of $c$ as quantifying how correlated different regions of a nematic profile are with one another and thus as some measure of the spatial structure of the system as it evolves over time. For further information on the importance of correlation functions in condensed matter physics, we recommend that the reader consult \textsc{Sethna} (\cite{Seth06}, chapter 10). An important point to be emphasised is that correlation functions are experimentally measurable by scattering experiments.

As mentioned above, it is argued in the physics literature (see for instance \textsc{Denniston et al.} \cite{PhysRevE.64.021701} or \textsc{Zapotocky et al.} \cite{zapotocky1}) that once a system is coerced into an isotropic-nematic phase transition, the correlation function assumes a self-similar form asymptotically in time. In summary, 
\begin{equation*}\label{corrfunsingle}
\langle c(r, t) \rangle \sim \Gamma\left(\frac{r}{L(t)}\right) \quad \text{in an appropriate topology as} \quad t\longrightarrow \infty,
\end{equation*}
where $\Gamma$ is some `universal' scaling function and $L(t)$ is the aforementioned characteristic length scale of nematic domains which invade the isotropic phase. With the interests of the experimental physicist still in mind, the angular brackets $\langle\cdot\rangle$ signify that one should consider a suitable average value of the correlation function computed over many repeated experimental trials, a process which `smoothes out' anomalous data gathered from experiment. If one then wishes information on the structure of this system for large values of time $t$, one need only rescale in the spatial variable $r$ to compute the appropriate value of $\langle c(r, t)\rangle$.

In order to model such an averaging procedure in a mathematically rigorous manner, we employ the notion of statistical solution of evolution equations introduced by \textsc{Foia\c{s}} \cite{MR0287390} and also by \textsc{Vishik and Fursikov} \cite{fursikov1}. This notion of solution to nonlinear PDE, which can simply be seen as a kind of measure-valued weak solution of a PDE, has been studied many authors in the mathematical theory of the Navier-Stokes equations to explore aspects of turbulence theory. We now outline the significance of statistical solutions in our study of phase transitions.
\subsection{`Statistical Self-similarity': Statistical Solutions of PDEs}\label{statisticalsimilarity}
One can show that the semiflow $\{S(t)\}_{t\geq 0}$ of the gradient flow \eqref{gradientflowings} of $E_{\mathrm{LdG}}$ possesses neither a global attractor nor an inertial manifold in the natural phase space $H:=L^{2}(\mathbb{R}^{3})\cap L^{6}(\mathbb{R}^{3})$, which is unfortunate as both objects are well suited to revealing asymptotic structure of dynamics. For more information on such attracting sets in PDE phase space problems, one could consult \textsc{Constantin, Foia\c{s}, Nicolaenko and Temam} \cite{MR966192} or \textsc{Robinson} \cite{MR1881888}. Since we are interested in studying the time asymptotics of the correlation function, this presents a significant challenge. Thus, acknowledging the structure of the ensemble of solution trajectories
\begin{equation*}
\mathcal{M}:=\bigcup_{Q_{0}\in H}\bigg\{\{Q(\cdot, t)\}_{t\geq 0}\,:\,Q(\cdot, 0)=S(0)Q_{0}=Q_{0}\bigg\},
\end{equation*}
as being rather complicated, we ask what one could say if one turns instead to the ensemble of measure-valued statistical solutions
\begin{equation*}
\mathcal{N}:=\bigcup_{\mu_{0}\in \mathsf{M}_{0}(H)}\bigg\{\{\mu_{t}\}_{t\geq 0}\,:\,\mu_{t}|_{t=0}=\mu_{0}\bigg\},
\end{equation*}
where $\mathsf{M}_{0}(H)$ denotes the set of all Borel probability measures on $H$.
The ensemble $\mathcal{N}$ is `coarser' than $\mathcal{M}$ in terms of information, in the sense that each such statistical solution $\mu_{t}$ of \eqref{gradientflowings} at time $t$ is insensitive to changes on sets of $\mu_{t}$-measure 0. In order to make use of this insensitivity to `fine structure' of the trajectory ensemble $\mathcal{M}$, and to characterise one qualitative aspect of solutions of \eqref{gradientflowings} on $H$, we rigorously define the correlation function $c$ previously introduced in \eqref{corryfunny} as
\begin{equation}\label{corrfun}
c_{\mu_{0}}(r, t):=\frac{\displaystyle\int_{H}\int_{\mathbb{R}^{3}}\mathrm{tr}\left(Q(x+r)Q(x)\right)\,dxd\mu_{t}(Q)}{\displaystyle\int_{H}\int_{\mathbb{R}^{3}}\mathrm{tr}\left(Q(x)^{2}\right)\,dxd\mu_{t}(Q)},
\end{equation}
where $\mu_{0}$ is a Borel probability measure on $H$. It is with this definition of the correlation function we work throughout this paper, as opposed to the earlier-given definition which is supported only on single trajectories in $\mathcal{M}$. Moreover, it allows us to give good mathematical sense to the average $\langle\cdot \rangle$ discussed above, namely $\langle c(r, t)\rangle:=c_{\mu_{0}}(r, t)$.

As we mentioned before, the gradient flow \eqref{gradientflowings} takes the shape of a nonlinear heat system for the evolution of nematic profiles $Q$. As this equation generates a strongly-continuous semigroup of solution operators $\{S(t)\}_{t\geq 0}$ on phase space $H$, it is possible to deal more concretely with the evolution of initial measures as opposed to the so-called Liouville equation formalism, which can be adopted when a semigroup of solution operators is not readily available: see \textsc{Foia\c{s} et al.} \cite{MR1855030}, chapter V for more information on this more general definition of statistical solution. To be precise, for a given suitable initial Borel measure $\mu_{0}$ supported on $H$, we construct in Section \ref{sectionfour} an associated one-parameter family of measures $\{\mu_{t}\}_{t\geq 0}$ with the property that
\begin{equation*}
\mu_{t}(E)=\lim_{k\rightarrow\infty}\sum_{j=1}^{N(k)}\vartheta_{j}^{(k)}\delta_{S(t)\overline{Q}_{j}^{(k)}}(E)
\end{equation*}
for measurable subsets $E\subseteq H$, nematic profiles $\overline{Q}_{j}^{(k)}\in H$ and $t\geq0$. One can then view each statistical solution $\{\mu_{t}\}_{t\geq 0}\in\mathcal{N}$ as a convenient concatenation of individual solutions $\{Q(\cdot, t)\}_{t\geq 0}\in\mathcal{M}$. 

In the study of long time behaviour of dynamical systems generated by evolution equations, one is often interested in constructing a {\em phase portrait} for all solution trajectories, including the location of any attracting sets of the dynamics. We store some remarks on the construction of a phase portrait for the correlation function $c_{\mu_{0}}$ in the closing section of this article.
\subsection{The Mathematical Model}\label{model}
We consistently employ the Einstein summation convention over repeated lower indices in the sequel. Let us now briefly discuss the main features of the models \eqref{ldg} and \eqref{gradientflowings} which make them suitable for the study of phase transitions.

In accordance with the Landau theory of phase transitions, the free energy density $f$ is assumed to be a smooth map which is frame indifferent, in our case invariant under the action of the group $\mathrm{O}(3)$:
\begin{equation}\label{hemitrope}
f\left(T_{\textsf{R}}D, \textsf{R}Q\textsf{R}^{T}\right)=f\left(D, Q\right),
\end{equation} 
where $Q\in \mathrm{Sym}_{0}(3)$ and $(T_{\textsf{R}}D)_{ijk}:=R_{i\ell}R_{jm}R_{kn}D_{\ell m n}$ for all $\textsf{R}\in\mathrm{O}(3)$ whenever $D$ is a rank 3 tensor. We assume that the corresponding free energy has the form 
\begin{equation*}
\int_{\mathbb{R}^{3}}f(\nabla Q(x), Q(x))\,dx =\int_{\mathbb{R}^{3}}\big(w(\nabla Q(x), Q(x))+f_{\mathrm{B}}(Q(x))\big)\,dx,
\end{equation*}
in which the elastic and bulk terms are decoupled. One may show by invariant theory that any such smooth isotropic map $f_{\mathrm{B}}$ is a smooth function of the principle invariants of $\mathrm{O}(3)$
\begin{equation*}
\mathrm{tr}\left(Q^{2}\right)\qquad\text{and}\qquad \mathrm{tr}\left(Q^{3}\right),
\end{equation*}
recalling that $\mathrm{tr}(Q)=0$. As stated above, we consider the energy functional $E_{\mathrm{LdG}}$, namely
\begin{equation*}
E_{\mathrm{LdG}}[Q]:=\int_{\mathbb{R}^{3}}\left(\frac{1}{2}|\nabla Q(x)|^{2}+\frac{a(\vartheta)}{2}\mathrm{tr}\left(Q^{2}\right)-\frac{b}{3}\mathrm{tr}\left(Q^{3}\right)+\frac{c}{4}\left(\mathrm{tr}\left(Q^{2}\right)\right)^{2}\right)\,dx,
\end{equation*} 
to which the elastic energy contribution is simply the Dirichlet energy, although one could work with more a general elastic energy of the form
\begin{equation*}
L_{1}|\nabla Q|^{2}+L_{2}Q_{ij, j}Q_{ik, k}+L_{3}Q_{ik, j}Q_{ij, k}+L_{4}Q_{k\ell}Q_{ij, k}Q_{ij, \ell},
\end{equation*} 
which is compatible with the symmetry \eqref{hemitrope} above. Moreover, we have chosen to employ the simplest bulk energy density $f_{\mathrm{B}}$ that is able to predict an isotropic-nematic phase transition, namely
\begin{equation}\label{bulk}
f_{\mathrm{B}}(Q):=\frac{a(\vartheta)}{2}\mathrm{tr}\left(Q^{2}\right)-\frac{b}{3}\mathrm{tr}\left(Q^{3}\right)+\frac{c}{4}\left(\mathrm{tr}\left(Q^{2}\right)\right)^{2},
\end{equation} 
where $a=\alpha(\vartheta-\vartheta_{\ast})$, $\vartheta_{\ast}>0$ is a critical temperature and $\vartheta\in\mathbb{R}$ is a temperature parameter controlling the depths of the nematic energy wells. In the low-temperature regime in which the nematic phase is energetically favourable, the material-dependent constants $(a, b, c)$ belong to the region of {\em uniaxial bistability}
\begin{equation*}
\mathcal{D}:=\left\{(a, b, c)\in (0, \infty)^{3}\,:\, b^{2}> 27ac\right\}
\end{equation*}  
of parameter space $(0, \infty)^{3}$. The model for the dynamics of nematic liquid crystals we study is the $L^{2}(\mathbb{R}^{3})$-gradient flow of the functional $E_{\mathrm{LdG}}$,
\begin{equation}\label{qteneq}
\frac{\partial Q}{\partial t}=\Delta Q - a(\vartheta)\,Q+b\left(Q^{2}-\frac{1}{3}\mathrm{tr}\left(Q^{2}\right)I\right)-c\,\mathrm{tr}\left(Q^{2}\right)Q.
\end{equation}
Furthermore, we model the temperature quench which renders the the isotropic phase less energically-favourable than the nematic as {\em instantaneous}, i.e. the parameters $(a, b, c)$ are fixed and lie in the region $\mathcal{D}$ and do not depend on an evolving temperature field $\vartheta$. Before the quench is effected at time $t=0$, one thinks of the parameters $(a, b, c)$ as lying in the complementary region $(0, \infty)^{3}\setminus \mathcal{D}$: in this case, the isotropic phase $Q\equiv 0$ is a global minimiser of the bulk energy density \eqref{bulk} above. Of course, one could consider the more physically-relevant problem where $\vartheta$ evolves under an appropriate evolution equation, however we believe this to be an unnecessary complication for investigating the asymptotic self-similarity of solutions $Q$.

We work over the whole space $\mathbb{R}^{3}$ as opposed to a bounded domain $\Omega\subset\mathbb{R}^{3}$ so we need not worry about restricting the spatial argument $r$ of the correlation function \eqref{corrfun}. One may interpret this in physical terms as modelling behaviour of a liquid crystal in the bulk of the material, far from any boundaries where complicated boundary effects could come into play.

The main result of this paper is the following:
\begin{thm}\label{mainy} For any given $\delta>0$ there exist $\eta>0$, depending only on $\delta$ and the parameters $(a, b, c)\in\mathcal{D}$, and an open dense subset of
\begin{equation*}
\left\{R\in L^{\infty}(\mathbb{R}^{3}, \symnn)\,:\, \underset{x\in\mathbb{R}^{3}}{\mathrm{ess}\sup}\left(1+|x|\right)^{8+\delta}|R(x)|< \eta\right\},
\end{equation*}
such that for any Borel probability measure supported in this open dense set, the associated correlation function $c_{\mu_{0}}$ \eqref{corrfun} exhibits asymptotic self-similar behaviour, viz
\begin{equation*}
\left\|c_{\mu_{0}}(r, t)- e^{-\frac{|r|^{2}}{8t}}\right\|_{L^{\infty}(\mathbb{R}^{3}, \,dr)}=\mathcal{O}\left(t^{-\frac{1}{2}}\right) \quad \text{as} \quad t\rightarrow\infty.
\end{equation*}
\end{thm}
In the final section of this article, we explore an alternative scaling regime for the correlation function $c_{\mu_{0}}(\cdot, t)$ as $t\rightarrow \infty$, when $\mu_{0}$ is supported on `large' initial data in $H$.
\begin{rem}
The structure of the open dense subset
\begin{equation*}
\left\{R\in L^{\infty}(\mathbb{R}^{3}, \symnn)\,:\, \underset{x\in\mathbb{R}^{3}}{\mathrm{ess}\sup}\left(1+|x|\right)^{8+\delta}|R(x)|< \eta\right\},
\end{equation*}
to which we refer in the statement of theorem \ref{mainy} is described in detail in section \ref{amatrix}.
\end{rem}
\subsection{Structure of the Paper} In section \ref{sectiontwo}, we record some preliminary results which allow us to establish a representation formula for `small initial data' solutions of the Q-tensor equation in section \ref{sectionthree}. The main result of the paper lies in section \ref{sectionfour}, in which we establish asymptotic self-similarity of the correlation function $c_{\mu_{0}}$ by means of the previously established representation formula. In the final section, we briefly discuss scaling behaviour of the correlation function corresponding to a class of solutions starting from `large' initial data, which gives rise to a different scaling regime for the correlation function, namely $L(t)=t$.
\subsection{Notation}
For any matrix $A\in\symnn$ we denote its Frobenius norm by $|A|:=\sqrt{A:A}$, where $A:B\equiv A_{ij}B_{ji}$ for $A, B\in \symnn$. In what follows, unless otherwise stated, all Lebesgue spaces have range in $\symnn$ and we write $L^{p}(\mathbb{R}^{3}, \,\symnn)$ simply as $L^{p}(\mathbb{R}^{3})$ for $1\leq p\leq\infty$. We also work with the phase space $H:=L^{2}(\mathbb{R}^{3})\cap L^{6}(\mathbb{R}^{3})$ endowed with the norm $\|Q\|:=\max\{\|Q\|_{2}, \|Q\|_{6}\}$, where the norms $\|\cdot\|_{p}$ are given by
\begin{displaymath}
\|Q\|_{p}:=\left\{
\begin{array}{ll}
\displaystyle \left(\int_{\mathbb{R}^{3}}\left(\mathrm{tr}\left[Q(x)^{2}\right]\right)^{\frac{p}{2}}\,dx\right)^{\frac{1}{p}} & \quad \text{when}\quad 1\leq p<\infty, \vspace{2mm} \\ \displaystyle \underset{x\in\mathbb{R}^{3}}{\mathrm{ess}\sup}\,|Q(x)| & \quad \text{when} \quad p=\infty.
\end{array}
\right.
\end{displaymath}
We denote the fundamental solution of the heat equation on $\mathbb{R}^{3}$ by $\Phi$, namely
\begin{equation*}
\Phi(x, t):=\frac{e^{-|x|^{2}/4t}}{(4\pi t)^{\frac{3}{2}}},
\end{equation*}
and denote the time-shifted profile $\Phi(x, t+1)$ simply by $\Phi_{1}(x, t)$. Finally, for any given $\delta>0$ we denote by $\mathcal{A}$ the space of essentially bounded functions given by
\begin{equation}\label{afunctions}
\mathcal{A}:=\left\{Q\in L^{\infty}(\mathbb{R}^{3})\,:\,\underset{x\in\mathbb{R}^{3}}{\mathrm{ess}\sup}\,\left(1+|x|\right)^{8+\delta}|Q(x)|<\infty\right\},
\end{equation}
and equip it with the natural norm $\|Q\|_{\mathcal{A}}:=\mathrm{ess}\sup_{x\in\mathbb{R}^{3}}(1+|x|)^{8+\delta}|Q(x)|$. The constant $\delta>0$ is chosen arbitrarily but remains \emph{fixed}, and we subsequently suppress the dependence of the space $\mathcal{A}$ on $\delta$.

\section{Some Preliminary Results}\label{sectiontwo}
We now establish some results which enable us to comment on the long-time behaviour of the correlation function concentrated on individual solutions of the Q-tensor equation, namely the asymptotic behaviour as $t\rightarrow\infty$ of the quantity
\begin{equation*}
c(r, t)=\frac{\displaystyle\int_{\mathbb{R}^{3}}\mathrm{tr}\left(Q(x, t)Q(x+r, t)\right)\,dx}{\displaystyle\int_{\mathbb{R}^{3}}\mathrm{tr}\left(Q(x, t)^{2}\right)\,dx},
\end{equation*}
where the initial datum $Q_{0}\in H$ is chosen appropriately. We begin by stating a simple result on the well-posedness of the Q-tensor equation in the phase space $H$. 
\begin{prop}\label{globexpro} 
For each $Q_{0}\in H$ there exists a unique global classically-smooth solution $Q$ of the Q-tensor equation \eqref{qteneq} on $\mathbb{R}^{3}\times (0, \infty)$ satisfying $\|Q(\cdot, t)\|<\infty$ for all $t>0$. Furthermore, the nonlinear semigroup of solution operators $S(t): H\rightarrow H$ defined by 
\begin{displaymath}
S(t)Q_{0}:=\left\{
\begin{array}{ll}
Q_{0}, & \quad \text{if} \quad t=0 \\ & \\ Q(\cdot, t; Q_{0}) & \quad \text{if} \quad t>0 
\end{array}
\right.
\end{displaymath}
for any $Q_{0}\in H$ is strongly continuous.
\end{prop}
\begin{proof}
The proof follows by standard techniques. We refer the reader to \textsc{Wilkinson} \cite{mythesis} for details.
\end{proof}
It is our aim to investigate finer properties of `small initial data' solutions of \eqref{qteneq}, namely we wish to demonstrate a decomposition result for solutions starting from initial data which are small in the $\|\cdot\|_{\mathcal{A}}$-norm. In order to achieve this, we first must examine the long-time behaviour of solutions of the heat equation in $L^{p}(\mathbb{R}^{3})$.

\subsection{A Remark on the Heat Equation in $\mathbb{R}^{d}$}
Consider the heat equation
\begin{equation*}
\frac{\partial u}{\partial t}(x, t)=\Delta u(x, t) \qquad (x, t)\in\mathbb{R}^{d}\times (0,\infty),
\end{equation*}
where $1\leq p\leq \infty$, $d\geq 1$ is an integer. For given $L^{q}(\mathbb{R}^{d})$ initial data, it possesses a classically smooth solution on $\mathbb{R}^{d}\times(0, \infty)$ and this solution $u(x, t)$ may be expressed as a convolution involving the $d$-dimensional heat kernel $\Phi_{d}(x, t)$, 
\begin{equation*}
u(x, t)=(\Phi_{d}(\cdot, t)\ast u_{0})(x)=\int_{\mathbb{R}^{d}}\frac{e^{-|x-y|^{2}/4t}}{(4\pi t)^{d/2}}u_{0}(y)\,dy,
\end{equation*}
for $t>0$. In order to obtain an estimate on the time decay of solutions in $L^{p}(\mathbb{R}^{d})$, one can employ Young's inequality in the above to produce
\begin{equation*}
\|u(\cdot, t)\|_{p}\leq \frac{C(r)}{t^{\frac{d}{2}\left(1-\frac{1}{r}\right)}}\|u_{0}\|_{q},
\end{equation*}
for all $t>0$ and $1+1/p=1/q+1/r$. Thus, all solutions of the heat equation with $L^{q}(\mathbb{R}^{d})$ initial data decay to zero \emph{at least} with rate $t^{-d/2+d/2r}$ as $t\rightarrow\infty$. We show in what is to come that if we insist the initial data of the heat equation are of \emph{zero mean} i.e. $\int_{\mathbb{R}^{d}}u_{0}=0$ and possess some mild decay condition at infinity, then we may improve our estimate on the rate of decay of the corresponding solution to 0 in $L^{p}(\mathbb{R}^{d})$ as $t\rightarrow\infty$.

We begin with an auxiliary pointwise estimate on solutions of the heat equation that will be of help to us in the following section, and may indeed be of independent interest. We consider only the case of the heat equation in $\mathbb{R}^{3}$ appropriate to our considerations. The methods used here may be extended without great effort to capture a similar result in $\mathbb{R}^{d}$ for arbitrary dimension $d\geq 1$.
\begin{prop}\label{maeeresult}
Suppose $\beta>0$, and consider an integrable map $u_{0}:\mathbb{R}^{3}\rightarrow\mathbb{R}$ satisfying
\begin{equation*}
\int_{\mathbb{R}^{3}}\left(1+|x|\right)^{\beta+1}|u_{0}(x)|\,dx<\infty.
\end{equation*}
The difference
$$\overline{m}(x, t):=(e^{t\Delta}u_{0})(x)-\Phi(x, t)\int_{\mathbb{R}^{3}}u_{0}(y)\,dy$$
satisfies the estimate
\begin{equation}\label{myresult1}
|\overline{m}(x, t)|\leq Ct^{-2}\left(1+\frac{|x|}{\sqrt{8t}}\right)^{-\beta}\int_{\mathbb{R}^{3}}|y|\left(1+\frac{|y|}{\sqrt{8t}}\right)^{\beta}|u_{0}(y)|\,dy
\end{equation}
for $x\in\mathbb{R}^{3}$, $t>0$. The explicitly-computable constant $C>0$ depends only on $\beta$.
\end{prop}
\begin{proof} Define the map $\sigma: [0, 2\pi)\times [0, \pi)\rightarrow\mathbb{S}^{2}$, a standard surface patch parameterisation on the unit sphere, to be 
\begin{equation*}
\sigma(\vartheta, \varphi):=\left(\cos{\vartheta}\sin{\varphi}, \sin{\vartheta}\sin{\varphi}, \cos{\varphi}\right),
\end{equation*}
and define $U:[0,\infty)\times [0,2\pi)\times [0, \pi)\rightarrow\mathbb{R}$ by
\begin{equation*}
U(r, \vartheta, \varphi):=\int_{r}^{\infty}u_{0}(\rho\cos\vartheta \sin\varphi, \rho\sin\vartheta \sin\varphi, \rho\cos\varphi)\rho^{2}\sin\varphi\,d\rho.
\end{equation*}
Now, by a change of co-ordinates and an application of integration by parts, we deduce that
\begin{align}
& \quad (e^{t\Delta}u_{0})(x) \notag\\ = & \quad \int_{0}^{\pi}\int_{0}^{2\pi}\int_{0}^{\infty}\frac{e^{-|x-r\sigma(\vartheta, \varphi)|^{2}/4t}}{(4\pi t)^{3/2}}u_{0}(r\sigma(\vartheta, \varphi))r^{2}\sin\varphi\,drd\vartheta d\varphi \notag\\ = & \quad \int_{\mathbb{R}^{3}} \Phi(x-r\sigma, t)\frac{\partial}{\partial r}\left(-U(r, \vartheta, \varphi)\right)\,drd\vartheta d\varphi \notag\\ = & \quad \Phi\int_{0}^{\pi}\int_{0}^{2\pi}U(0, \vartheta, \varphi)\, d\vartheta d\varphi +  \int_{\mathbb{R}^{3}}\frac{\sigma\cdot(x-r\sigma)}{2t}\Phi(x-r\sigma, t)U(r, \vartheta, \varphi)\,drd\vartheta d\varphi \notag\\ = & \quad \Phi\int_{\mathbb{R}^{3}}u_{0}(y)\,dy + \int_{\mathbb{R}^{3}}\frac{\sigma\cdot(x-r\sigma)}{2t}\Phi(x-r\sigma, t)U(r, \vartheta, \varphi)\,drd\vartheta d\varphi.\notag
\end{align}
Applying the elementary inequality $\sqrt{z}\exp(-z)\leq\exp(-z/2)$ for $z\geq 0$, we deduce from the above
\begin{align}
& \quad (e^{t\Delta}u_{0})(x)-\Phi(x, t)\int_{\mathbb{R}^{3}}u_{0}(y)\,dy \notag\\ \leq & \quad \frac{1}{(4\pi)^{3/2}}\frac{1}{t^{2}}\int_{\mathbb{R}^{3}} e^{-|x-r\sigma|^{2}/8t}\Bigg(\int_{r}^{\infty}|u_{0}(\rho\sigma)|\rho^{2}d\rho\Bigg)\,\sin\varphi \,drd\vartheta d\varphi \notag\\ \leq & \quad Ct^{-2}\int_{\mathbb{R}^{3}}\sup_{(\vartheta, \varphi)}\left(e^{-|x-r\sigma|^{2}/8t}\right)\Bigg(\int_{r}^{\infty}|u_{0}(\rho\sigma)|\rho^{2}d\rho\Bigg)\,\sin\varphi \,drd\vartheta d\varphi  \notag\\ = &  \quad Ct^{-2}\int_{\mathbb{R}^{3}} e^{-(|x|-r)^{2}/8t}\Bigg(\int_{r}^{\infty}|u_{0}(\rho\sigma)|\rho^{2}d\rho\Bigg)\,\sin\varphi \,drd\vartheta d\varphi, \notag
\end{align}
where $C>0$ is independent of $t>0$. A couple of applications of Fubini's theorem then yield
\begin{eqnarray}
\overline{m}(x, t) & \leq & Ct^{-2}\int_{0}^{\infty}e^{-(|x|-r)^{2}/8t}\int_{|y|\geq r}|u_{0}(y)|\,dy\,dr \notag\\ & = & Ct^{-2}\int_{\mathbb{R}^{3}}\left(\int_{0}^{\infty}(1-\chi_{B(0, r)}(y))e^{-(|x|-r)^{2}/8t}\,dr\right)|u_{0}(y)|\,dy, \notag
\end{eqnarray}
where $\chi_{B(0, r)}$ denotes the characteristic function of the open ball $B(0, r)\subset\mathbb{R}^{3}$. From this we find
\begin{equation}\label{returntome}
\overline{m}(x, t)\leq Ct^{-2}\int_{\mathbb{R}^{3}}\left(\int_{0}^{|y|}e^{-(|x|-r)^{2}/8t}\,dr\right)|u_{0}(y)|\,dy.
\end{equation}
We pause at this point to obtain a helpful estimate on the $y$-integrand in the above.
\begin{lem}
For any $\beta>0$, there exists a constant $C=C(\beta)>0$ such that
\begin{equation}\label{mineineq}
\int_{X-Y}^{X}e^{-\xi^{2}}d\xi\leq CY\left(\frac{1+Y}{1+X}\right)^{\beta} \quad \text{for}\hspace{2mm}X, Y\geq 0.
\end{equation}
\end{lem}
\begin{proof} In what follows we fix $Y\geq 0$, considering it as a parameter. We separate the demonstration of the inequality into two cases.
\newline\noindent\textit{(i) The case where $0\leq X\leq Y$:} One may easily deduce from the chain $0\leq X\leq Y$ the simple inequality
\begin{equation*}
Y\leq Y\left(\frac{1+Y}{1+X}\right)^{\beta},
\end{equation*}
for any $\beta>0$. Using this, we have
\begin{equation*}
\int_{X-Y}^{X}e^{-\xi^{2}}\,d\xi\leq\int_{X-Y}^{X}\max_{\xi_{0}}\left(e^{-\xi_{0}^{2}}\right)\,d\xi=Y\leq Y\left(\frac{1+Y}{1+X}\right)^{\beta}.
\end{equation*}
Thus, the claimed inequality holds for $0\leq X\leq Y$.
\newline\noindent\textit{(ii) The case where $Y\leq X$:} This time, we have the estimate
\begin{equation*}
\int_{X-Y}^{X}e^{-\xi^{2}}\,d\xi\leq\int_{X-Y}^{X}\sup_{\xi_{0}}\left(e^{-\xi_{0}^{2}}\right)\,d\xi=Ye^{-(X-Y)^{2}}.
\end{equation*}
Now, since there exists a constant $C>1$ dependent only on $\beta>0$ such that
\begin{equation*}
e^{-Z^{2}}\leq C\left(\frac{1}{1+Z}\right)^{\beta},
\end{equation*}
for all $Z\geq 0$, and noting that
\begin{equation*}
\left(\frac{1}{1+Z}\right)^{\beta}\leq \left(\frac{1+Y}{1+Y+Z}\right)^{\beta},
\end{equation*}
for all $Y\geq 0$, it follows from the substitution $Z\mapsto X-Y$ that
\begin{equation*}
e^{-(X-Y)^{2}}\leq C\left(\frac{1+Y}{1+X}\right)^{\beta}.
\end{equation*} 
for all $X\geq Y$. This estimate together with the first case estimate yields the result \eqref{mineineq} above. 
\end{proof}
We now return to the estimate \eqref{returntome} and close the proof of the proposition. Transforming the $y$-integrand through a simple change of variables provides
\begin{eqnarray}\label{final}
\overline{m}(x, t) & \leq & Ct^{-2}\int_{\mathbb{R}^{3}}\left(\int_{0}^{|y|}e^{-(|x|-r)^{2}/8t}\,dr\right)|u_{0}(y)|\,dy \notag \\ & = & Ct^{-2}\int_{\mathbb{R}^{3}}\left(\sqrt{8t}\int_{\frac{|x|}{\sqrt{8t}}-\frac{|y|}{\sqrt{8t}}}^{\frac{|x|}{\sqrt{8t}}}e^{-\xi^{2}}\,d\xi\right)|u_{0}(y)|\,dy.
\end{eqnarray} 
Finally, setting $X:=|x|/(8t)^{1/2}$ and $Y:=|y|/(8t)^{1/2}$, we apply the above lemma to \eqref{final}, from which estimate \eqref{myresult1} follows. 
\end{proof}
\begin{rem}
Raising both sides of \eqref{myresult1} to the power $p>1$ and integrating over the domain $\mathbb{R}^{3}$, one has that
\begin{equation*}
\left\|e^{t\Delta}u_{0}-\Phi(\cdot, t)\int_{\mathbb{R}^{3}}u_{0}(y)\,dy\right\|_{p}\leq \frac{C(p, \beta)}{t^{2-\frac{3}{2p}}}\int_{\mathbb{R}^{3}}|u_{0}(y)|(1+|y|)^{\beta+1}\,dy,
\end{equation*}
for $t>0$. Therefore, if the initial data of the heat equation are of mean zero and lie in $L^{1}(\mathbb{R}^{3}; d\omega_{\beta})$, where $d\omega_{\beta}:=(1+|x|)^{\beta+1}dx$, we have improved upon the standard estimate obtained by Young's inequality for the time decay of solutions in $L^{p}(\mathbb{R}^{3})$. Estimates of type \eqref{myresult1} are key to obtaining a representation formula in the following section for solutions of \eqref{qteneq}. Let us also remark that the above result leads to a generalisation to higher spatial dimension of the work of \textsc{Taskinen} {\em \cite{Task1}}.
\end{rem}

\section{A Representation Formula}\label{sectionthree}
We establish here a helpful structural formula for solutions of the gradient flow starting from `small' initial data. In particular, we show that the leading order asymptotics (up to order $t^{-3/2}$ in time after rescaling) of such solutions in all $L^{p}$-norms are governed by a matrix multiple of the heat kernel. It proves important in our subsequent analysis to consider a transformed version of the gradient flow \eqref{qteneq}. If we set $R(x, t):=e^{at}Q(x, t)$ and substitute $R$ into \eqref{qteneq} in a formal manner, we obtain the equation
\begin{equation}\label{req}
\partial_{t}R=\Delta R+b\,e^{-at}\left(R^{2}-\frac{1}{3}\mathrm{tr}\left(R^{2}\right)\right)-c\,e^{-2at}\mathrm{tr}\left(R^{2}\right)R,
\end{equation}
which is the principle equation of analysis in the section to come. 

Let $\delta>0$ be arbitrarily chosen but {\em fixed}. We recall from \eqref{afunctions} that the space $\mathcal{A}$ is defined to be 
\begin{equation*}
\left\{Q\in L^{\infty}(\mathbb{R}^{3})\,:\, \underset{x\in\mathbb{R}^{3}}{\mathrm{ess \,sup}}\left(1+|x|\right)^{8+\delta}|Q(x)|<\infty\right\}.
\end{equation*}
Let $X_{0}$ denote the set of maps 
\begin{equation*}
X_{0}:=\left\{W\in C(\mathbb{R}^{3}\times [0, \infty); \,\symnn)\, \left|\, \sup_{(x, t)} \,\omega(x, t)|W(x, t)|<\infty\right\}\right.,
\end{equation*}
where the space-time weight $\omega$ is defined to be
\begin{equation*}
\omega(x, t):=\left(1+\frac{|x|}{\sqrt{t+1}}\right)^{4+\frac{\delta}{2}}(t+1)^{2}.
\end{equation*}
When endowed with the natural norm
\begin{equation*}
\|W\|_{X_{0}}:=\sup_{(x, t)} \,\omega(x, t)|W(x, t)|,
\end{equation*}
one may verify that $X_{0}$ admits the structure of a Banach space. We also employ the product space $X:=\symnn\times X_{0}$ furnished with the norm $\|\cdot\|_{X}:=|\cdot| + \|\cdot\|_{X_{0}}$.
\begin{rem}
The space of maps $X_{0}$ with which we work from now on might seem at first glance rather ungainly. However, this space is only auxiliary in the sense that it does not feature in the hypotheses of theorem \ref{thm1} below. In our setting, it is essentially the `right' space in which to capture the first correction to the leading heat kernel asymptotics of small initial datum solutions of the transformed equation \eqref{req}.
\end{rem}

The main result of this section is contained in the following theorem.
\begin{thm}\label{thm1}
There exists $\eta>0$ depending only on $(a, b, c)\in\mathcal{D}$ such that for initial data $Q_{0}\in\mathcal{A}\subset H$ satisfying the smallness condition
\begin{equation}\label{smallness}
\|Q_{0}\|_{\mathcal{A}}\leq \eta,
\end{equation} 
the Q-tensor equation \eqref{qteneq} has a unique smooth solution on $\mathbb{R}^{3}\times (0, \infty)$ of the form
\begin{equation}\label{decomposition}
Q(x, t)=Ae^{-at}\Phi_{1}(x, t)+e^{-at}V(x, t),
\end{equation}
where $A\in\symnn$ and $V\in X_{0}$.
\end{thm}
\subsection{An Auxiliary System of Integral Equations}\label{rougharg}
With the above-defined Banach spaces in mind, for any given initial data $Q_{0}\in \mathcal{A}$ consider the following nonlinear operator $F:X\rightarrow X$ associated with $Q_{0}$ defined componentwise as $F_{1}:X\rightarrow\symnn$, where
\begin{equation}\label{fonecomponent}
F_{1}(A, V):=\int_{\mathbb{R}^{3}}Q_{0}(y)\,dy+\int_{0}^{\infty}\int_{\mathbb{R}^{3}}h(A\Phi_{1}(y, s)+V(y, s), s)\,dyds,
\end{equation}
and $F_{2}:X\rightarrow X_{0}$, where
\begin{align}\label{ftwocomponent}
F_{2}(A, V; \cdot, t) \quad := & \quad e^{t\Delta}\left(Q_{0}-\frac{e^{-|\cdot|^{2}/4}}{(4\pi)^{3/2}}\int_{\mathbb{R}^{3}}Q_{0}(y)\,dy\right) \notag \\ & \quad + \int_{0}^{t}e^{(t-s)\Delta}h(A\Phi_{1}(\cdot, s) + V(\cdot, s), s)\,ds \notag \\ & \quad - \Phi_{1}(\cdot, t)\int_{0}^{\infty}\int_{\mathbb{R}^{3}}h(A\Phi_{1}(y, s)+V(y, s), s)\,dyds.
\end{align}
The function $h$ is the nonlinearity in the transformed Q-tensor equation \eqref{req}, and is given by
\begin{equation*}
h(R, t):=b\,e^{-at}\left(R^{2}-\frac{1}{3}\mathrm{tr}\left(R^{2}\right)I\right)-c\,e^{-2at}\mathrm{tr}\left(R^{2}\right)R,
\end{equation*}
for $R\in\symnn$ and $t\geq 0$.

Suppose that there exists a fixed point $(A^{\ast}, V^{\ast})$ of the operator $F$ in $X$, i.e. that the integral equations
\begin{equation*}
A^{\ast}=F_{1}(A^{\ast}, V^{\ast})\quad \text{and} \quad V^{\ast}=F_{2}(A^{\ast}, V^{\ast})
\end{equation*}
hold in $\symnn$ and $X_{0}$, respectively. Multiplying throughout the first equality above by $\Phi_{1}(x, t)$ and then adding the contribution of the second equality, one finds that the integral equation
\begin{equation*}
A^{\ast}\Phi_{1}(\cdot, t)+V^{\ast}(\cdot, t) = e^{t\Delta}Q_{0} + \int_{0}^{t}e^{(t-s)\Delta}h(A^{\ast}\Phi_{1}+ V^{\ast}, s)\,ds
\end{equation*}
holds for the assumed fixed point $(A^{\ast}, V^{\ast})$. Writing $R^{\ast}:=A^{\ast}\Phi_{1}+ V^{\ast}$, we note that the above integral equation implies that $R^{\ast}$ is classically smooth and satisfies \eqref{req} on $\mathbb{R}^{3}\times (0, \infty)$. Following the time-dependent transformation effected by $Q^{\ast}(x, t) = e^{-at}R^{\ast}(x, t)$, one immediately finds that $Q^{\ast}$ satisfies \eqref{qteneq} on the same space-time domain. In turn, we may deduce from proposition \ref{globexpro} (uniqueness of solutions) that for the chosen initial datum $Q_{0}\in \mathcal{A}\subset H$, the corresponding solution $Q^{\ast}$ of the Q-tensor equation admits the decomposition 
\begin{equation*}
Q^{\ast}(x, t)=A^{\ast}e^{-at}\Phi_{1}(x, t)+e^{-at}V^{\ast}(x, t),
\end{equation*} 
for $t>0$ as claimed. Therefore, in order to achieve this representation formula it suffices to show that the operator $F:X\rightarrow X$ possesses a fixed point in $X$ for given $Q_{0}\in \mathcal{A}$.  We achieve this for a special class of $Q_{0}$ (namely those which satisfy $\|Q_{0}\|_{\mathcal{A}}<\eta$ for $\eta>0$ sufficiently small) through an application of Banach's fixed point theorem. Establishing the fact that $F$ is locally Lipschitz on $X$ is a little delicate, and we consider details of this calculation below.
\subsection{The Operator $F:X\rightarrow X$ is Locally Lipschitz}
In order to make our presentation more concise, we adopt the following notation. For given $A, B\in\symnn$ and $V, W \in X_{0}$, we define the \emph{contraction modulus} $k_{0}\equiv k_{0}(A, B; V, W)$ to be
\begin{equation*}
k_{0}:=\left(|A| + |B| + \| V\|_{X_{0}}+\| W\|_{X_{0}}\right)+\left(|A| + |B| + \| V\|_{X_{0}}+\| W\|_{X_{0}}\right)^{2},
\end{equation*}
and write $\phi\lesssim\psi$ whenever $\phi\leq Ck_{0}\psi$ and $C>0$ is some absolute constant. Moreover, to avoid any confusion which may arise from line to line, we interpret the string $\phi\lesssim\psi\lesssim\omega$ as $\phi\leq C_{1}k_{0}\psi\leq C_{2}k_{0}\omega$ for some absolute constants $C_{1}, C_{2}>0$. We demonstrate that the operator $F$ satisfies the difference property
\begin{equation*}
\| F(A, V) - F(B, W)\|_{X}\lesssim \| (A, V) - (B, W)\|_{X},
\end{equation*}
for any $A, B\in\symnn$ and $V, W\in X_{0}$, following which we deduce that $F$ is indeed well-defined as an operator on and with range in $X$ and thus is locally Lipschitz on $X$. Indeed, we gather these goals together in the following proposition.
\begin{prop}
The operator $F$ associated with $Q_{0}\in\mathcal{A}$ defined on $X$ component-wise in \eqref{fonecomponent} and \eqref{ftwocomponent} is well defined, i.e. has range in $X$, and also satisfies the local Lipschitz condition
\begin{equation}\label{lipschitz}
\| F(A, V) - F(B, W)\|_{X} \lesssim \|(A, V) - (B, W)\|_{X},
\end{equation}
whenever $(A, V)$ and $(B, W)$ belong to some bounded subset of $X$.
\end{prop}
\begin{proof}
We begin by stating a useful estimate upon which we call frequently during the course of our proof.
\begin{lem} For $A, B\in\symnn$, $V, W\in X_{0}$, $x\in\mathbb{R}^{3}$ and $t\geq 0$, we have
\begin{align}\label{hestimate}
& \left|h(A\Phi_{1}(x, t)+V(x, t), t)-h(B\Phi_{1}(x, t)+W(x, t), t)\right| \notag \\ \lesssim \quad & \left(1+\frac{|x|}{\sqrt{t+1}}\right)^{-8-\delta}e^{-at}(t+1)^{-3}\|(A, V)-(B, W)\|_{X}.
\end{align}
\end{lem} 
\begin{proof}
The proof is skipped and is left to the reader. 
\end{proof}
We now look to obtain estimates on the components of $F$ and its differences in the appropriate spaces. 
\subsubsection*{The component $F_{1}:X\rightarrow \symnn$}
The treatment of $F_{1}$ is straightforward: following an application of the estimate \eqref{hestimate}, we deduce that the inequality
\begin{equation}\label{flipest}
|F_{1}(A, V)-F_{1}(B, W)|\leq C_{1}k_{0}(A, B; V, W)\|(A, V) - (B, W)\|_{X}
\end{equation}
holds for any $A, B\in\symnn$ and $V, W\in X_{0}$ and constant $C_{1}>0$. Furthermore, one can quickly verify that $\trr(F_{1}(A, V))=0$ and indeed $F_{1}(A, V)<\infty$ for any $(A, V)\in X$. 

We now turn our attention to the component $F_{2}$, whose analysis requires considerably more care than that of $F_{1}$. In order to deal with the singularity of the heat kernal at $t=0$, we cleave our estimates into one part which is in a neighbourhood of the origin in time, and one other part which is bounded away from zero in time. We now consider the more involved case when time $t$ is bounded away from the origin.
\subsubsection*{The component $F_{2}:X\rightarrow X_{0}$ when $t\geq 1$}
A judicious splitting of integrals faciliates a swift analysis. For $(A, V), (B, W)\in X$, we begin by writing the difference as the sum of two pieces, namely
\begin{equation*}
F_{2}(A, V; x, t) - F_{2}(B, V; x, t) = I(x, t)+J(x, t),
\end{equation*}
where 
\begin{align*}
I(\cdot, t):=\int_{t/2}^{t}\hkern\big( h(A\Phi_{1}(\cdot, s)+V(\cdot, s), s)-h(B\Phi_{1}(\cdot, s)+W(\cdot, s), s)\big)\,ds,
\end{align*}
and
\begin{align*}
J(\cdot, t):= \int_{0}^{t/2}\hkern\big( h(A\Phi_{1}+V, s)-h(B\Phi_{1}+W, s)\big)\,ds\\ \quad -\Phi_{1}(\cdot, t)\int_{0}^{\infty}\int_{\mathbb{R}^{3}}\big(h(A\Phi_{1}+V, \tau)-h(B\Phi_{1}+W, \tau)\big)\,dyd\tau.
\end{align*}
We deal initially with the estimate on the term $I$ and consider at first its temporal integrand, given by
\begin{equation}\label{back}
\int_{\mathbb{R}^{3}}\frac{e^{-|x-y|^{2}/4(t-s)}}{(4\pi(t-s))^{3/2}}\left(h(A\Phi_{1}+V, s)-h(B\Phi_{1}+W, s)\right)\,dy
\end{equation} 
for $x\in \mathbb{R}^{3}$, $t\geq 1$ and $t/2\leq s\leq t$. For any given fixed $x\in\mathbb{R}^{3}$ the inequality 
\begin{equation}\label{auxineqone}
\left(1+\frac{|y|}{\sqrt{s+1}}\right)^{-8-\delta}\leq C\left(1+\frac{|x|}{\sqrt{t+1}}\right)^{-8-\delta}
\end{equation}
holds for all $\delta>0$, $t/2\leq s\leq t$ and $y$ in the region $\mathcal{R}_{1}(x)$, where
\begin{equation*}
\mathcal{R}_{1}(x):=\left\{y\in\mathbb{R}^{3}\,:\,|x-y|\leq \frac{|x|}{2}\right\},
\end{equation*}
and the constant $C>0$ depends only on $\delta>0$. For a fixed $x\in\mathbb{R}^{3}$, we consider the integral \eqref{back} but over the region $\mathcal{R}_{1}(x)$ as opposed to the whole space:
\begin{align}\label{goone}
&\int_{\mathcal{R}_{1}(x)}\Phi(x-y, t-s)\left(h(A\Phi_{1} + V, s) - h(B\Phi_{1} + W, s)\right)\,ds \notag \\ \overset{\eqref{hestimate}}{\lesssim} & \hspace{2mm }\frac{\|(A, V)-(B, W))\|_{X}}{e^{as}(s+1)^{3}}\int_{\mathcal{R}_{1}(x)}\Phi(x-y, t-s)\left(1+\frac{|y|}{\sqrt{s+1}}\right)^{-8-\delta}\,dy\notag \\ \overset{ \eqref{auxineqone}}{\lesssim} & \hspace{2mm} \frac{\|(A, V)-(B, W))\|_{X}}{e^{as}(s+1)^{3}}\left(1+\frac{|x|}{\sqrt{t+1}}\right)^{-8-\delta}\int_{\mathcal{R}_{1}(x)}\Phi(x-y, t-s)\,dy\notag \\  \lesssim & \hspace{2mm} \frac{\|(A, V)-(B, W))\|_{X}}{e^{as}(s+1)^{3}}\left(1+\frac{|x|}{\sqrt{t+1}}\right)^{-8-\delta},
\end{align}
using the fact that the heat kernel has unit mass over $\mathbb{R}^{3}$. For the same fixed $x\in\mathbb{R}^{3}$, the inequality
\begin{equation}\label{auxexp}
\exp\left(-\frac{1}{8}\frac{|x-y|^{2}}{(t-s)}\right)\leq C\left(1+\frac{|x|}{\sqrt{t+1}}\right)^{-4-\frac{\delta}{2}}
\end{equation}
holds for $y$ in the region $\mathcal{R}_{2}(x):=\mathbb{R}^{3}\setminus\mathcal{R}_{1}(x)$, whenever $t/2\leq s\leq t$. Now proceeding in a similar manner as for the region $\mathcal{R}_{1}(x)$, we find that
\begin{align}\label{gotwo}
& \int_{\mathcal{R}_{2}(x)}\Phi(x-y, t-s)\left(h(A\Phi_{1} + V, s) - h(B\Phi_{1} + W, s)\right)\,dy \notag \\ = & \hspace{2mm} \int_{\mathcal{R}_{2}(x)}\frac{e^{-|x-y|^{2}/8(t-s)}}{(4\pi(t-s))^{3/2}}e^{-|x-y|^{2}/8(t-s)}\left(h(A\Phi_{1} + V, s) - h(B\Phi_{1} + W, s)\right)\,dy \notag \\ \overset{\eqref{auxexp}}{\leq} & \hspace{2mm} C\left(1+\frac{|x|}{\sqrt{t+1}}\right)^{-4-\frac{\delta}{2}}\times \notag \\ & \int_{\mathcal{R}_{2}(x)}\frac{e^{-|x-y|^{2}/8(t-s)}}{(4\pi(t-s))^{3/2}}\left|h(A\Phi_{1} + V, s) - h(B\Phi_{1} + W, s)\right|\,dy & \notag \\ \overset{\eqref{hestimate}}{\lesssim} & \frac{\|(A, V)-(B, W))\|_{X}}{e^{as}(s+1)^{3}}\left(1+\frac{|x|}{\sqrt{t+1}}\right)^{-4-\frac{\delta}{2}}\int_{\mathcal{R}_{2}(x)}\frac{e^{-|x-y|^{2}/8(t-s)}}{(4\pi(t-s))^{3/2}}\,dy.
\end{align}
Combining results \eqref{goone} and \eqref{gotwo} we obtain
\begin{align}\label{joneest}
& \displaystyle\int_{\mathbb{R}^{3}}\Phi(x-y, t-s)\left(h(A\Phi_{1}(y, s) + V(y, s), s) - h(B\Phi_{1}(y, s) + W(y, s), s)\right)\,dy \notag \\ \lesssim & \displaystyle \hspace{2mm} \frac{1}{e^{as}(s+1)^{3}}\left(1+\frac{|x|}{\sqrt{t+1}}\right)^{-4-\frac{\delta}{2}}\|(A, V)-(B, W))\|_{X}.
\end{align}
Finally, integrating \eqref{joneest} with respect to $s$ over the interval $t/2\leq s \leq t$, we find
\begin{equation}\label{finaljone}
I(x, t) \lesssim \left(1+\frac{|x|}{\sqrt{t+1}}\right)^{-4-\frac{\delta}{2}}(t+1)^{-2}\|(A, V) - (B, W)\|_{X},
\end{equation}
as desired. 

Let us progress to the required estimate for $J$. It is at this point we utilise the heat equation `decay improvement' proposition \ref{maeeresult}. For convenience, for our chosen $(A, V), (B, W)\in X$ we denote by $h_{0}:\mathbb{R}^{3}\times (0, \infty)\rightarrow\symnn$ the map
\begin{equation*}
h_{0}(x, t):=h(A\Phi_{1}(x, t)+V(x, t), t)-h(B\Phi_{1}(x, t)+W(x, t), t).
\end{equation*}
If we denote by $\overline{M}(x, t, s)$ the quantity
\begin{equation*}
\overline{M}(\cdot, t, s):=\int_{\mathbb{R}^{3}}\frac{e^{-|\cdot-y|^{2}/4(t-s)}}{(4\pi(t-s))^{3/2}}h_{0}(y, s)\,dy-\Phi(\cdot, t-s)\left(\int_{\mathbb{R}^{3}}h_{0}(y, s)\,dy\right)
\end{equation*}
and also stipulate that $H_{0}:(0, \infty)\rightarrow \symnn$ be
\begin{equation*}
H_{0}(t):=\int_{t}^{\infty}\int_{\mathbb{R}^{3}}h_{0}(y, \tau)\,dyd\tau,
\end{equation*}
one may verify that the identity
\begin{align*}
J(x, t) = & \overbrace{\int_{0}^{t/2}\overline{M}(x, t, s)\,ds}^{J_{1}(x, t):=} + \overbrace{H_{0}(0)\big(\Phi(x, t)-\Phi_{1}(x, t)\big)}^{J_{2}(x, t):=} -\overbrace{\Phi\left(x, \frac{t}{2}\right)H_{0}\left(\frac{t}{2}\right)}^{J_{3}(x, t):=} \notag \\ \quad & +\underbrace{\int_{0}^{t/2}\left(\frac{6\pi}{t-s}-\frac{|x|^{2}}{4(t-s)^{2}}\right)\Phi(x, t-s)H_{0}(s)\,ds}_{J_{4}(x, t):=}
\end{align*}
holds for all $x\in\mathbb{R}^{3}$ and $t\geq 1$. 
It is straightforward to show that
\begin{equation}\label{itwotofive}
J_{i}(x, t) \lesssim \left(1+\frac{|x|}{\sqrt{t+1}}\right)^{-4-\frac{\delta}{2}}(t+1)^{-2}\|(A, V) - (B, W)\|_{X},
\end{equation}
for $i=3, 4$ simply by using familiar properties of the heat kernel, inequality \eqref{hestimate} and the fact that $(t-s)^{-1}\leq 4(t+1)^{-1}$ for $t\geq 1$ and $s$ in the interval $0\leq s\leq t/2$. The treatment of the cases $i=1, 2$ require a little more care. An application of the result of proposition \ref{maeeresult} to the term $J_{1}$ yields
\begin{align*}
& \quad \int_{0}^{t/2}\overline{M}(x, t, s)\,ds \\ \overset{\eqref{myresult1}}{\leq} & \quad C \int_{0}^{t/2}\int_{\mathbb{R}^{3}}\frac{|y|}{(t-s)^{2}}\left(\frac{1+|y|/\sqrt{8(t-s)}}{1+|x|/\sqrt{8(t-s)}}\right)^{4+\frac{\delta}{2}}\times \\ & \quad \left|h(A\Phi_{1} + V, s) - h(B\Phi_{1} + W, s)\right|\,dyds \\ \overset{\eqref{hestimate}}{\lesssim} & \quad \int_{0}^{t/2}\frac{e^{-as}}{(s+1)^{3}}\int_{\mathbb{R}^{3}}\frac{|y|}{(t-s)^{2}}\left(1+\frac{|y|}{\sqrt{8(t-s)}}\right)^{4+\frac{\delta}{2}}\left(1+\frac{|y|}{\sqrt{s+1}}\right)^{-8-\delta}\,dyds \\ & \qquad \times \left(1+\frac{|x|}{\sqrt{8(t-s)}}\right)^{-4-\frac{\delta}{2}}\|(A, V)-(B, W))\|_{X}.
\end{align*}
Now, noticing that for $0\leq s\leq t/2$ and $t\geq 1$ we have the string of inequalities $(t-s)^{-1/2}\leq \sqrt{2}t^{-1/2}\leq C(t+1)^{-1/2}\leq C(s+1)^{-1/2}$, we deduce from the above that
\begin{align}\label{ione}
& \quad \int_{0}^{t/2}\overline{M}(x, t, s)\,ds \notag\\ \lesssim \quad & \int_{0}^{t/2}\left(1+\frac{|x|}{\sqrt{8(t-s)}}\right)^{-4-\frac{\delta}{2}}\frac{e^{-as}}{(s+1)}\frac{1}{(t-s)^{2}}\left(\int_{\mathbb{R}^{3}}(1+|y|)^{-3-\frac{\delta}{2}}\,dy\right)\,ds \notag \\ & \qquad \times\|(A, V)-(B, W))\|_{X}\notag \\ \lesssim \quad & \left(1+\frac{|x|}{\sqrt{t+1}}\right)^{-4-\frac{\delta}{2}}(t+1)^{-2}\|(A, V)-(B, W))\|_{X},
\end{align}
since $\delta>0$ gives us the required integrability of $(1+|\cdot|)^{-3-\delta/2}$ in dimension three. It remains to verify such an estimate holds for the term $J_{2}$, and for this we require the following simple lemma. 
\begin{lem}\label{heatkerdiff} For $x\in \mathbb{R}^{3}$ and $t\geq 1$, the heat kernel difference $\Phi_{1}(x, t)-\Phi(x, t)$ satisfies the inequality
\begin{equation}\label{heatdiff}
\left|\Phi_{1}(x, t)-\Phi(x, t)\right|\leq \frac{2 e^{-|x|^{2}/8(t+1)}}{(t+1)^{5/2}}.
\end{equation}
\end{lem}
\begin{proof} Consider the smooth map $\psi$ defined by
\begin{equation*}
\psi(z, t):= e^{-\frac{z^{2}}{4}\left(\frac{1}{2(t+1)}\right)}-e^{-\frac{z^{2}}{4}\left(\frac{2+t}{2t(t+1)}\right)}.
\end{equation*}
For any fixed $t\geq 1$, by considering the equation $\psi_{z}(z, t)=0$ one may show that $z\mapsto\psi(z, t)$ is controlled by its unique global maximum over the set $[1, \infty)$, namely
\begin{equation*}
\psi(z, t)\leq 2 \left(\frac{1}{\frac{2}{t}+1}\right)^{t/2}\frac{1}{t+1}.
\end{equation*}
Upon setting $z=|x|$, we readily deduce from the above 
\begin{equation*}
e^{-|x|^{2}/4(t+1)}-e^{-|x|^{2}/4t}\leq \frac{2e^{-|x|^{2}/8(t+1)}}{(t+1)},
\end{equation*}
from which estimate \eqref{heatdiff} quickly follows. 
\end{proof}
Piecing this together with the estimates \eqref{itwotofive} and \eqref{ione}, we obtain
\begin{equation*}
J(x, t) \lesssim \|(A, V) - (B, W)\|_{X}\left(1+\frac{|x|}{\sqrt{t+1}}\right)^{-4-\frac{\delta}{2}}(t+1)^{-2},
\end{equation*}
which along with \eqref{finaljone} provides
\begin{align}\label{ftwotgeqone}
& \quad \left|F_{2}(A, V; x, t) - F_{2}(B, W; x, t)\right| \notag \\ \lesssim & \quad \|(A, V) - (B, W)\|_{X}\left(1+\frac{|x|}{\sqrt{t+1}}\right)^{-4-\frac{\delta}{2}}(t+1)^{-2},
\end{align}
for $x\in\mathbb{R}^{3}$ and $t\geq 1$. 
\subsubsection*{The component $F_{2}:X\rightarrow X_{0}$ when $0\leq t \leq 1$} This case follows almost immediately from \eqref{hestimate} and the fact that our time interval of interest in this instance is compact. Thus, combining this with \eqref{ftwotgeqone}, rearranging and taking norms in $X_{0}$, we find that
\begin{equation}\label{ftwo}
\|F_{2}(A, V) - F_{2}(B, W)\|_{X_{0}}\leq C_{2}k_{0}(A, B; V, W)\|(A, V) - (B, W)\|_{X}
\end{equation}
for some constant $C_{2}>0$. We have therefore verified the difference property \eqref{lipschitz}.
\subsubsection*{The operator $F_{2}:X\rightarrow X_{0}$ is well defined}
All that remains to be checked is that the map
\begin{equation*}
(x, t) \mapsto \int_{\mathbb{R}^{3}}\frac{e^{-|x-y|^{2}/4t}}{(4\pi t)^{3/2}}\left(Q_{0}(y)-\frac{e^{-|y|^{2}/4}}{(4\pi)^{3/2}}\int_{\mathbb{R}^{3}}Q_{0}(z)\,dz\right)\,dy
\end{equation*}
is an element of $X_{0}$. The case when $t\geq 1$ may be tackled using inequality \eqref{myresult1} and lemma \ref{heatkerdiff}. On the other hand, for the case when $0\leq t\leq 1$, we need only concern ourselves with the behaviour of the heat term $e^{t\Delta}Q_{0}$ near the origin in time. We obtain the required estimates by writing
\begin{equation*}
(e^{t\Delta}Q_{0})(x)=\int_{\mathcal{R}_{1}(x)}\Phi(x-y, t)Q_{0}(y)\,dy+\int_{\mathcal{R}_{2}(x)}\Phi(x-y, t)Q_{0}(y)\,dy
\end{equation*} 
and applying inequalities \eqref{auxineqone} and \eqref{auxexp} in the manner previously outlined over the regions $\mathcal{R}_{1}(x)$ and $\mathcal{R}_{2}(x)$, respectively. 

Finally, setting $(B, W)=0$ in \eqref{ftwo} completes the proof that $F:X\rightarrow X$ is both well-defined and locally Lipschitz on $X$.
\end{proof}
Let $\mathfrak{B}:=\overline{B(0, \varepsilon_{0})}\subset X$ denote the closed ball of radius $\varepsilon_{0}>0$ in $X$, where $\varepsilon_{0}$ is yet to be fixed. To close the proof of theorem \ref{thm1} by means of an application of Banach's fixed point theorem, it remains to show that $F: \mathfrak{B}\rightarrow \mathfrak{B}$ is a strictly contractive operator for $\varepsilon_{0}$ chosen sufficiently small. To this end, fix $\varepsilon_{0}$ to be the positive root of the quadratic equation associated with the contraction constraint
\begin{equation*}
\max\left\{C_{1}, C_{2}\right\}\left(2\varepsilon+4\varepsilon^{2}\right)=\frac{1}{4},
\end{equation*}
where $C_{1}$ and $C_{2}$ are those constants appearing in \eqref{flipest} and \eqref{ftwo}, respectively. Taking into account the definition of the contraction modulus $k_{0}$, we deduce from \eqref{flipest} and \eqref{ftwo} that
\begin{equation*}
\|F(A, V)-F(B, W)\|_{X}\leq \frac{1}{2}\|(A, V) - (B, W)\|_{X},
\end{equation*}
for all $(A, V), (B, W)\in \mathfrak{B}$. Given that we also have the estimate
\begin{align*}
& \quad \|F(A, V)\|_{X} \\ \leq & \quad \sup_{(x, t)}\left(1+\frac{|x|}{\sqrt{t+1}}\right)^{4+\frac{\delta}{2}}(t+1)^{2}\left|e^{t\Delta}\left(Q_{0}-\frac{e^{-|\cdot|^{2}/4}}{(4\pi)^{3/2}}\int_{\mathbb{R}^{3}}Q_{0}(z)\,dz\right)(x)\right| \\ \quad & + \left|\int_{\mathbb{R}^{3}}Q_{0}(y)\,dy\right| + \frac{1}{2}\|(A, V)\|_{X},
\end{align*}
for any $(A, V)\in\mathfrak{B}$, it is clear we may find $\eta>0$ to ensure that the quantity
\begin{align*}
& \quad \sup_{(x, t)}\left(1+\frac{|x|}{\sqrt{t+1}}\right)^{4+\frac{\delta}{2}}(t+1)^{2}\left|e^{t\Delta}\left(Q_{0}-\frac{e^{-|\cdot|^{2}/4}}{(4\pi)^{3/2}}\int_{\mathbb{R}^{3}}Q_{0}(z)\,dz\right)(x)\right| \\ + & \quad  \left|\int_{\mathbb{R}^{3}}Q_{0}(y)\,dy\right|
\end{align*}  
is less than or equal to $\varepsilon_{0}/2$ for all $Q_{0}\in\mathcal{A}$ satisfying $\|Q_{0}\|_{\mathcal{A}}\leq \eta$. Thus, for all such $Q_{0}$ satisfying this `smallness' condition, the associated nonlinear operators $F$ are strictly contractive with range in $\mathfrak{B}$. By the Banach fixed point theorem, there exists a unique (relabeled) fixed point $(A, V)\in \mathfrak{B}$ and so from the discussion in section \ref{rougharg} follows the proof of theorem \ref{thm1}, namely
\begin{equation*}
Q(x, t)=Ae^{-at}\Phi_{1}(x, t)+e^{-at}V(x, t),
\end{equation*}
for some $A\in\symnn$ and $V\in X_{0}$ satisfying $\|(A, V)\|_{X}\leq \varepsilon_{0}$.
\begin{rem}
We henceforth denote by $\mathcal{A}^{\circ}\subset\mathcal{A}$ the open subset of all initial data $Q_{0}\in\mathcal{A}$ for which the above decomposition result holds, namely
\begin{equation*}\label{acirc}
\mathcal{A}^{\circ}:=\left\{R\in L^{\infty}(\mathbb{R}^{3})\,:\, \underset{x\in\mathbb{R}^{3}}{\mathrm{ess}\sup}\left(1+|x|\right)^{8+\delta}|R(x)|< \eta\right\}.
\end{equation*}
\end{rem}
\section{The Scaling Regime $L(t)=t^{1/2}$}\label{sectionfour}
Before we discuss the behaviour of the correlation function $c_{\mu_{0}}(r, t)$ as $t\rightarrow\infty$, we comment on the behaviour of the correlation function concentrated on \emph{individual} solutions. Utilising the decomposition formula $$Q(x, t)=Ae^{-at}\Phi_{1}(x, t)+e^{-at}V(x, t)$$ obtained in theorem \ref{thm1}, we investigate how the quantity
\begin{equation*}
c(r, t)=\frac{\displaystyle\int_{\mathbb{R}^{3}}\mathrm{tr}\left(Q(x+r, t)Q(x, t)\right)\,dx}{\displaystyle\int_{\mathbb{R}^{3}}\mathrm{tr}\left(Q(x, t)^{2}\right)\,dx}
\end{equation*}
behaves for large time. Firstly, noting that we have the equality
\begin{equation*}
\int_{\mathbb{R}^{3}}\Phi_{1}(x+r, t)\Phi_{1}(x, t)\,dx=\frac{1}{\sqrt{8}}\frac{e^{-|r|^{2}/8(t+1)}}{(4\pi(t+1))^{3/2}},
\end{equation*}
and also that any $V\in X_{0}$ has the norm decay property
\begin{equation*}
\|V(\cdot, t)\|_{2} \leq C\|V\|_{X_{0}}(t+1)^{-5/4},
\end{equation*}
one may then verify that
\begin{equation}\label{veryhelpful}
\frac{\sqrt{8}e^{2at}(4\pi(t+1))^{3/2}}{\mathrm{tr}(A^{2})}\int_{\mathbb{R}^{3}}\mathrm{tr}\left(Q(x+r, t)Q(x, t)\right)\,dx \leq e^{-|r|^{2}/8(t+1)}+\omega(t)
\end{equation}
and 
\begin{equation}\label{alsoveryhelpful}
\frac{\mathrm{tr}(A^{2})}{\sqrt{8}e^{2at}(4\pi(t+1))^{3/2}}\left(\int_{\mathbb{R}^{3}}\mathrm{tr}\left(Q(x, t)^{2}\right)\,dx\right)^{-1} \leq \frac{1}{1+\omega(t)},
\end{equation}
for all $t>0$ sufficiently large and $r\in \mathbb{R}^{3}$, \emph{provided} $A\neq 0$, where the function $\omega$ is bounded and continuous on $[1, \infty)$ and decays at least as quickly at $t^{-1/2}$ as $t\rightarrow \infty$. Taking the product of \eqref{veryhelpful} with \eqref{alsoveryhelpful} above, one quickly finds after subtracting off the contribution $\exp(-|r|^{2}/8t)$ that
\begin{align*}
& \quad c(r, t)-e^{-|r|^{2}/8t}\\ \leq & \quad  \frac{1}{1+\omega(t)}\left(e^{-|r|^{2}/8(t+1)}-e^{-|r|^{2}/8t}\right) - \frac{\omega(t)}{1+\omega(t)}e^{-|r|^{2}/8t}+\frac{\omega(t)}{1+\omega(t)},
\end{align*}
from which we deduce 
\begin{equation}\label{corrindiv}
\left\|c(r, t)-e^{-\frac{|r|^{2}}{8t}}\right\|_{L^{\infty}(\mathbb{R}^{3};\,dr)}=\mathcal{O}\left(t^{-1/2}\right) \quad \text{as} \quad t\longrightarrow \infty.
\end{equation}
Let us emphasise once more that the above calculation is only valid whenever the constant matrix $A$ is non-zero. If $A=0$, one discovers that the correlation function has the form
\begin{equation*}
\displaystyle c(r, t)=\frac{\displaystyle \int_{\mathbb{R}^{3}}\mathrm{tr}\left(V(x+r, t)V(x, t)\right)\,dx}{\displaystyle \int_{\mathbb{R}^{3}}\mathrm{tr}\left(V(x, t)^{2}\right)\,dx},
\end{equation*}
from which no scaling information on $c(r, t)$ can be gleaned. This is, of course, not the case when $A\neq 0$ as the map $(r, t)\mapsto e^{-|r|^{2}/8t}$ is manifestly self-similar on $\mathbb{R}^{3}\times (0, \infty)$. We now demonstrate that the set of all initial data in $\mathcal{A}^{\circ}$ which possibly give rise to solutions for which $A=0$ in the decomposition \eqref{decomposition} is rare, in the sense that it is a closed subset of $\mathcal{A}^{\circ}$ containing no open $\|\cdot\|_{\mathcal{A}}$-ball, i.e. it has empty interior. In particular, the complement of this set in $\mathcal{A}^{\circ}$ is dense in $\mathcal{A}^{\circ}$.
\subsection{The matrix $A$ cannot be zero on `large' sets of initial data in $\mathcal{A}^{\circ}$}\label{amatrix}
It is our aim to show that the scaling behaviour of solutions is generic amongst all those evolving from initial data satisfying the smallness condition \eqref{smallness}. Let us reiterate that our notion of `generic' here is that the set of all such initial data in $\mathcal{A}^{\circ}$ giving rise to solutions with constant matrix $A=0$ constitute a closed set containing no open ball. 

Firstly, we note that as each $Q_{0}\in \mathcal{A}^{\circ}$ gives rise to a unique solution of the shape
\begin{equation}\label{decomp}
Q(x, t)=Ae^{-at}\Phi_{1}(x, t)+e^{-at}V(x, t),
\end{equation}
where each matrix $A$ satisfies the explicit identity (c.f. formula \eqref{fonecomponent})
\begin{equation*}
A=\int_{\mathbb{R}^{3}}Q_{0}(y)\,dy + \int_{0}^{\infty}\int_{\mathbb{R}^{3}}h(Q(y, \tau), \tau)\,dyd\tau,
\end{equation*}
we may ask about the nature of the resulting map between $\mathcal{A}^{\circ}$ and $\symnn$ which yields $A$ from $Q_{0}$. It is well known that real-analytic maps between open connected sets of Banach spaces cannot be constant on open subsets unless they are identically constant on the whole set: see, for instance, \textsc{Buffoni and Toland} \cite{buffoni1}, theorem 4.3.9. Thus, if we demonstrate the map $Q_{0}\mapsto A$ is real-analytic on $\mathcal{A}^{\circ}$ and is indeed not identically the zero map, it follows that the set of all initial data which produce the result $A=0$ in the decomposition \eqref{decomp} is closed in $\mathcal{A}^{\circ}$ and contains no $\|\cdot\|_{\mathcal{A}}$-open ball.

In this direction, consider the map $\mathcal{F}_{0}:\mathcal{X}\times \mathcal{A}\rightarrow \mathcal{X}$ defined by
\begin{equation*}
\mathcal{F}_{0}(Z(\cdot, t), z):=e^{t\Delta}z+\int_{0}^{t}e^{(t-s)\Delta}h(Z(\cdot, s), s)\,ds-Z(\cdot, t),
\end{equation*}
for $Z\in\mathcal{X}$, $z\in\mathcal{A}$ and $t\geq 0$, where $\mathcal{X}$ is the Banach space of continuous maps
\begin{equation*}
\mathcal{X}:=\left\{U:\,\sup_{(x, t)}\left(1+\frac{|x|}{\sqrt{t+1}}\right)^{4+\frac{\delta}{2}}(t+1)^{3/2}|U(x, t)|<\infty\right\}
\end{equation*}
endowed with the natural weighted supremum norm $\|\cdot\|_{\mathcal{X}}$. We also recall that the space $\mathcal{A}$ is endowed with the norm $\|R\|_{\mathcal{A}}=\text{ess sup}_{x\in\mathbb{R}^{3}}(1+|x|)^{8+\delta}|R(x)|$. In what follows, we wield results from the theory of analytic Banach space-valued maps. We refer the reader to \textsc{Zeidler} \cite{MR816732} for the basic definitions and results. As a helpful first step, we establish the following proposition.
\begin{prop}
The map $\mathcal{F}_{0}:\mathcal{X}\times \mathcal{A}\rightarrow \mathcal{X}$ is real-analytic.
\end{prop}
\begin{proof}
By considering the difference 
\begin{equation*}
\mathcal{F}_{0}[(R, \rho)+(\Xi, \xi)]-\mathcal{F}_{0}[(R, \rho)],
\end{equation*}
one may check that the first Fr\'{e}chet derivative $d\mathcal{F}_{0}$ of $\mathcal{F}_{0}$ at $(R, \rho)\in\mathcal{X}\times\mathcal{A}$ satisfies
\begin{align*}\label{homeomorph}
(d\mathcal{F}_{0}[(R, \rho)])(\Xi, \xi)=e^{t\Delta}\xi+ b\int_{0}^{t}e^{(t-s)\Delta}e^{-as}\left(R\Xi+\Xi R-\frac{2}{3}\mathrm{tr}\left(R\Xi\right)I\right)\,ds\notag \\ -\, c\int_{0}^{t}e^{(t-s)\Delta}e^{-2as}\left(2\,\mathrm{tr}\left(R\Xi\right)R+\mathrm{tr}\left(R^{2}\right)\Xi\right)\,ds - \Xi(\cdot, t),
\end{align*}
for all $(\Xi, \xi)\in\mathcal{X}\times\mathcal{A}$. Owing to the fact that the nonlinearity $h$ in the operator $\mathcal{F}_{0}$ is of cubic order, the higher Fr\'{e}chet derivatives satisfy
\begin{equation*}
(d^{k}\mathcal{F}[(R, \rho)])\left((\Xi_{1}, \xi_{1}), ..., (\Xi_{k}, \xi_{k})\right)=0 \quad \text{in}\quad\mathcal{X},
\end{equation*}
for $(R, \rho), (\Xi_{1}, \xi_{1}) ..., (\Xi_{k}, \xi_{k})\in\mathcal{X}\times\mathcal{A}$ and $k\geq 4$. Furthermore, for constants $C, K$ and $r$ depending on the point $(R, \rho)$, the estimates
\begin{equation*}
\|d^{k}\mathcal{F}_{0}[(Z, z)]\|\leq \frac{Ck!}{K^{k}}, \quad \text{whenever}\quad \|(Z, z)-(R, \rho)\|_{\mathcal{X}\times\mathcal{A}}<r
\end{equation*}
hold for $0\leq k < 4$, where norms are taken in the appropriate space of multilinear operators. By triviality of the higher Fr\'{e}chet derivatives,
the map $\mathcal{F}_{0}$ is real-analytic on $\mathcal{X}\times\mathcal{A}$.
\end{proof}
Having established that the map $\mathcal{F}_{0}:\mathcal{X}\times\mathcal{A}\rightarrow\mathcal{X}$ is analytic, this result aids us in establishing the useful fact that the solution map $\mathcal{F}: Q_{0}\mapsto Q$ from $\mathcal{A}^{\circ}$ to $\mathcal{X}$ is itself a real-analytic map. 
\begin{thm}\label{analy}
The solution map $\mathcal{F}:\mathcal{A}^{\circ}\rightarrow\mathcal{X}$ is real-analytic.
\end{thm}
\begin{proof}
In what follows, we call $(R, \rho)\in \mathcal{X}\times\mathcal{A}$ a \emph{solution pair} if and only if $R\in \mathcal{X}$ is the unique solution of \eqref{qteneq} corresponding to the initial datum $\rho\in\mathcal{A}$. We denote the corresponding \emph{solution map} which yields $R$ from $\rho$ by $\mathcal{F}$. We make the important observation that for any point $(R, \rho)\in\mathcal{X}\times \mathcal{A}$,
\begin{equation}\label{usefulf}
\mathcal{F}_{0}[(R, \rho)]=0\qquad \text{if and only if} \qquad (R, \rho) \,\,\text{is a solution pair.}
\end{equation}
Suppose $(R, \rho)$ chosen from the open subset $\mathcal{X}\times\mathcal{A}^{\circ}\subset\mathcal{X}\times \mathcal{A}$ is a solution pair. We claim that the partial derivative $d_{Z}\mathcal{F}_{0}[(R, \rho)]\in\mathcal{L}(\mathcal{X})$ is a homeomorphism. 

We prove the claim as follows. As the partial Fr\'{e}chet derivative $d_{Z}\mathcal{F}_{0}$ satisfies
\begin{align*}
(d_{Z}\mathcal{F}_{0}[(R, \rho)])(\Xi)= b\int_{0}^{t}e^{(t-s)\Delta}e^{-as}\left(R\Xi+\Xi R-\frac{2}{3}\mathrm{tr}\left(R\Xi\right)I\right)\,ds\\ -\, c\int_{0}^{t}e^{(t-s)\Delta}e^{-2as}\left(2\,\mathrm{tr}\left(R\Xi\right)R+\mathrm{tr}\left(R^{2}\right)\Xi\right)\,ds - \Xi(\cdot, t)
\end{align*}
for all $\Xi\in\mathcal{X}$, one may then verify that the inequality
\begin{equation*}
\| d_{Z}\mathcal{F}_{0}[(R, \rho)]- J\|_{L(\mathcal{X})}<1
\end{equation*}
holds by `smallness' in the $\|\cdot\|_{\mathcal{X}}$-norm of those solutions of \eqref{req} starting from $\mathcal{A}^{\circ}$ initial data. Furthermore, the map $J:\mathcal{X}\rightarrow \mathcal{X}$ given by $J\Xi:=-\Xi$ clearly has unit norm in $\mathcal{L}(\mathcal{X})$. It then follows from the theory of Neumann series that $d_{Z}\mathcal{F}_{0}[(R, \rho)]\in L(\mathcal{X})$ is a homeomorphism. Thus, by the analytic implicit function theorem, 
one may infer the existence of an open neighbourhood $\mathcal{V}\subset\mathcal{A}^{\circ}$ of $\rho$, an open neighbourhood $\mathcal{U}\subset\mathcal{X}\times\mathcal{A}^{\circ}$ of the point $(R, \rho)$ and a real-analytic map $\phi: \mathcal{V}\rightarrow\mathcal{X}$ with the property that
\begin{equation*}
\left\{(\phi(\rho), \rho)\,:\,\rho\in\mathcal{V}\right\}=\mathcal{F}_{0}^{-1}(0)\cap\mathcal{U}.
\end{equation*}
By observation \eqref{usefulf}, it is clear that the set $\mathcal{F}_{0}^{-1}(0)$ is simply that of all solution pairs $(\mathcal{F}\rho, \rho)$ such that $\rho\in\mathcal{A}$. Therefore, the map $\phi:\mathcal{V}\rightarrow\mathcal{X}$ is identically equal to the restriction $\mathcal{F}|_{\mathcal{V}}$ of the solution operator to the set $\mathcal{V}\subset\mathcal{A}^{\circ}$.

As the point $(R, \rho)$ was arbitrarily chosen, it then follows that the solution operator $\mathcal{F}:\mathcal{A}^{\circ}\rightarrow\mathcal{X}$ is real-analytic. 
\end{proof}
Now that we have established the basic result that the solution operator $\mathcal{F}:\mathcal{A}^{\circ}\rightarrow\mathcal{X}$ is real-analytic, we define some other relevant operators. We define $\mathcal{G}:\mathcal{A}^{\circ}\rightarrow\mathcal{A}^{\circ}\times\mathcal{X}$ by
\begin{equation*}
\mathcal{G}Q_{0}:=\left(Q_{0}, \mathcal{F}Q_{0}\right)\qquad\text{for all} \qquad Q_{0}\in\mathcal{A}^{\circ}.
\end{equation*}
By virtue of the last theorem, it is clear that $\mathcal{G}$ is real-analytic on $\mathcal{A}^{\circ}$. Moreover, by a routine calculation, one may also verify that the map $\mathcal{H}:\mathcal{A}^{\circ}\times \mathcal{X}\rightarrow \mathrm{Sym}_{0}(3)$ defined by 
\begin{equation*}
\mathcal{H}(z, Z):=\int_{\mathbb{R}^{3}}z(y)\,dy+\int_{0}^{\infty}\int_{\mathbb{R}^{3}}h(Z(y, \tau), \tau)\,dyd\tau
\end{equation*}
is also real analytic. With these facts in mind, one may view the right-hand side of the equality
\begin{equation*}
A=\int_{\mathbb{R}^{3}}Q_{0}(y)\,dy + \int_{0}^{\infty}\int_{\mathbb{R}^{3}}h(Q(y, \tau), \tau)\,dyd\tau
\end{equation*}
as a composition of analytic maps taking $\mathcal{A}^{\circ}$ to $\symnn$, namely
\begin{equation*}
A=(\mathcal{H}\circ\mathcal{G})Q_{0}.
\end{equation*} 
As compositions of real analytic maps are themselves real analytic (see \textsc{Buffoni and Toland} \cite{buffoni1}, theorem 4.5.7),
we deduce that the map $\mathcal{H}\circ\mathcal{G}:Q_{0}\mapsto A$ is real analytic on $\mathcal{A}^{\circ}$. 

We need only now show that the map $\mathcal{H}\circ\mathcal{G}$ is non-zero on $\mathcal{A}^{\circ}$. To this end, we provide a non-empty set of initial data contained in $\mathcal{A}^{\circ}$ for which the matrix $A$ in the decomposition \eqref{decomposition} is non-zero. Together with the fact that the coefficient matrix $A$ depends in a continuous manner on initial data, we have the existence of an open set of initial data for which $A\neq 0$.

\begin{prop}
For $\alpha\geq 0$, let $\lambda_{\alpha}:\mathbb{R}^{3}\rightarrow \mathbb{R}$ denote the map
\begin{equation}\label{lambdalpha}
\lambda_{\alpha}(x):=-\frac{\alpha}{(1+|x|)^{8+\delta}},
\end{equation}
and set $Q_{0, \alpha}:=\mathrm{diag}(\lambda_{\alpha}, \lambda_{\alpha}, -2\lambda_{\alpha})$. There exists $\alpha_{0}>0$ such that $A\neq 0$ in the decomposition \eqref{decomposition} for all $0<\alpha<\alpha_{0}$.
\end{prop}
\begin{proof}
By uniqueness of solutions of the gradient flow \eqref{qteneq}, it follows that solutions starting from uniaxial data 
\begin{displaymath}
Q_{0}(x)=\left(
\begin{array}{ccc}
\lambda_{0}(x) & 0 & 0\\
0 & \lambda_{0}(x) & 0 \\
0 & 0 & -2\lambda_{0}(x)
\end{array}
\right)
\end{displaymath} remain of uniaxial form 
\begin{displaymath}
Q(x, t)=\left(
\begin{array}{ccc}
\lambda(x, t) & 0 & 0\\
0 & \lambda(x, t) & 0 \\
0 & 0 & -2\lambda(x, t)
\end{array}
\right)
\end{displaymath} 
for $t>0$, where $\lambda(x, t)$ satisfies the scalar nonlinear heat equation
\begin{equation}\label{lambdaheat}
\frac{\partial\lambda}{\partial t}=\Delta \lambda - a\,\lambda-b\,\lambda^{2}-6c\,\lambda^{3}
\end{equation}
on $\mathbb{R}^{3}\times (0, \infty)$. Suppose that we consider solutions of \eqref{lambdaheat} starting from initial data of the form \eqref{lambdalpha}. By \textsc{Theorem} \ref{thm1}, there exists $\overline{\alpha}>0$ such that if $|\lambda_{0}(x)|<\alpha<\overline{\alpha}$, then
\begin{equation*}
|\lambda(x, t)|\leq C\alpha \quad \text{on} \quad \mathbb{R}^{3}\times (0, \infty),
\end{equation*}
for some constant $C>0$ depending only on $(a, b, c)\in \mathcal{D}$.

Let us make the transformation $\nu(x, t):=e^{at}\lambda(x, t)$. Now, one can find $\alpha_{0}>0$ small enough satisfying $\overline{\alpha}\geq\alpha_{0}>0$ such that
\begin{equation*}
-b\,e^{-at}\left(\nu(x, t)\right)^{2}-c\,e^{-2at}\left(\nu(x, t)\right)^{3}\leq 0 \quad \text{on} \quad \mathbb{R}^{3}\times (0, \infty).
\end{equation*} 
Furthermore, since $\nu$ can easily be shown to satisfy the equation
\begin{equation*}
\nu(\cdot, t)=e^{t\Delta}\lambda_{0}-\int_{0}^{t}e^{(t-s)\Delta}\left(b\,e^{-as}\left(\nu(\cdot, s)\right)^{2}+6c\,e^{-2as}\left(\nu(\cdot, s)\right)^{3}\right)\,ds,
\end{equation*}
by integrating across over space, we deduce that
\begin{equation*}
\int_{\mathbb{R}^{3}}\nu(x, t)\,dx\leq\int_{\mathbb{R}^{3}}\lambda_{0}(x, t)\,dx<0,
\end{equation*}
for all $t>0$. Thus, the $L^{1}(\mathbb{R}^{3})$-norm of the original solutions $\lambda$ cannot decay to $0$ as $t\rightarrow\infty$, and so for the set of initial data
\begin{equation*}
\left\{\mathrm{diag}(\lambda_{0}, \lambda_{0}, -2\lambda_{0})\,:\, \lambda_{0}(x)=-\frac{\alpha}{(1+|x|)^{8+\delta}}\quad \text{for}\quad 0<\alpha<\alpha_{0}\right\}\subset\mathcal{A}^{\circ},
\end{equation*}
the corresponding matrix $A$ in the solution decomposition \eqref{decomposition} cannot be zero.  
\end{proof}
To complete our demonstration that $\mathcal{H}\circ\mathcal{G}$ is not identically the zero map on $\mathcal{A}^{\circ}$, we argue as follows. For $\alpha$ chosen from the interval $(0, \alpha_{0})$ determined from the above proposition, we denote the coefficient matrix in the decomposition formula \eqref{decomposition} corresponding to the solution with uniaxial initial datum $Q_{0, \alpha}=\mathrm{diag}(\lambda_{\alpha}, \lambda_{\alpha}, -2\lambda_{\alpha})$ by $A(\alpha)$. Since $A(\alpha)\neq 0$, there exists an open ball of matrices $\mathbb{B}(\alpha)\subset\mathrm{Sym}_{0}(3)$ with centre $A(\alpha)$ that does not contain the zero matrix. Since $\mathcal{H}\circ\mathcal{G}$ is analytic, it follows that $(\mathcal{H}\circ\mathcal{G})^{-1}(\mathbb{B}(\alpha))$ is an open set of initial data in $\mathcal{A}^{\circ}$. From this we conclude that $\mathcal{H}\circ\mathcal{G}$ cannot be the zero map on $\mathcal{A}^{\circ}$, from which it follows that the set of all those $Q_{0}$ in $\mathcal{A}^{\circ}$ which yield $A=0$ in \eqref{decomp} is closed and contains no open ball.
\begin{rem}\label{bremark}
We subsequently denote the set of all maps in $\mathcal{A}^{\circ}$ for which the coefficient matrix $A$ is 0 by $\mathcal{B}$.
\end{rem}

\subsection{Asymptotic Behaviour of $c_{\mu_{0}}(r, t)$ as $t\rightarrow\infty$}\label{whatdaf}
As discussed in the introduction, we express the notion of averaging over initial conditions in a rigorous manner by evaluating the correlation function with respect to a suitable time-dependent family of Borel probability measures. We construct statistical solutions of \eqref{qteneq} in the following proposition, whose proof is a modification of a construction contained in \textsc{Foias, Manley, Rosa and Temam} \cite{MR1855030}. We quickly recall the definition a Dirac delta measure on the space $L^{2}(\mathbb{R}^{3})$.
\begin{defn}
Let $\mathcal{P}(L^{2}(\mathbb{R}^{3}))$ denote the power set of $L^{2}(\mathbb{R}^{3})$. For any given $Q\in L^{2}(\mathbb{R}^{3})$, the associated map $\delta_{Q}:\mathcal{P}(L^{2}(\mathbb{R}^{3}))\rightarrow \{0, 1\}$ defined by
\begin{displaymath}
\delta_{Q}(A):=\left\{
\begin{array}{ll}
1 & \quad \text{if} \quad Q\in A, \\ & \\ 0 & \quad \text{otherwise},
\end{array}
\right.
\end{displaymath}
is the \emph{Dirac delta measure} concentrated at $Q\in L^{2}(\mathbb{R}^{3})$.
\end{defn}
\begin{prop}
Suppose $0<\gamma < \infty$. For any given Borel probability measure $\overline{\mu}$ supported on the set
\begin{equation*}
\left\{Q\in L^{2}(\mathbb{R}^{3})\,:\, \|Q\|_{2}\leq \gamma \right\}\cap H,
\end{equation*} 
there exists a corresponding one-parameter family of Borel probability measures $\{\mu_{t}\}_{t\geq 0}$ and an approximating sequence of families
\begin{equation*}
\left\{\{\mu_{t}^{(k)}\}_{t\geq 0}\,:\, k=1, 2, 3, ...\right\} \quad \text{with} \quad \mu_{t}^{(k)}:=\sum_{j=1}^{N(k)}\vartheta_{j}^{(k)}\delta_{S(t)\overline{Q}_{j}^{(k)}}
\end{equation*} 
for some $N(k)\in\mathbb{N}$ and $\vartheta_{j}^{(k)}\in (0, 1]$, such that
\begin{equation}\label{meas1}
\mu_{t}(E)=\lim_{k\rightarrow\infty}\mu_{t}^{(k)}(E)
\end{equation}
for all $t\geq 0$ and all $\mu_{t}$-measurable subsets $E\subseteq \{Q\in L^{2}(\mathbb{R}^{3})\,:\, \|Q\|_{2}\leq \gamma \}\cap H$. For each $k\geq 1$, the coefficients $\vartheta_{j}^{(k)}\geq 0$ respect the sum $\sum_{j}\vartheta_{j}^{(k)}=1$, and $S(t)\overline{Q}_{j}^{(k)}$ denotes the action of the semigroup of proposition \ref{globexpro} on initial data $\overline{Q}_{k, j}$.
\end{prop} 
\begin{proof}
We establish the spaces which we use to construct the family of measures $\{\mu_{t}\}_{t\geq 0}$. We endow the set
\begin{equation*}
K:=\left\{Q\in L^{2}(\mathbb{R}^{3})\,:\,\|Q\|_{2}\leq \gamma\right\}
\end{equation*} 
with the weak topology inherited from $L^{2}(\mathbb{R}^{3})$, with respect to which it is a compact topological space. Furthermore, this space is metrisable with metric $d_{w}$, given by
\begin{equation*}
d_{w}(Q, R):=\sum_{k=1}^{\infty}\frac{1}{2^{k}}\left|\int_{\mathbb{R}^{3}}Q:\chi_{k} - \int_{\mathbb{R}^{3}}R:\chi_{k}\right|,
\end{equation*} 
where $\{\chi_{k}\}_{k=1}^{\infty}$ is a countable dense subset of the set $K\subset L^{2}(\mathbb{R}^{3})$. We write $\textsf{M}_{0}(K)$ to denote the set of all Borel probability measures carried by subsets of $K$, and furnish $\textsf{M}_{0}(K)$ with the subspace weak-$\ast$ topology inherited from $C(K)'$, whence $\textsf{M}_{0}(K)$ is a convex, compact Hausdorff topological space by the Banach-Alaoglu theorem. 

Let $T>0$ be given. The energy estimate 
\begin{equation*}
\frac{1}{2}\frac{d}{dt}\int_{\mathbb{R}^{3}}\mathrm{tr}\left(Q^{2}\right)\leq \left(\frac{b^{2}}{2c}-a\right)\int_{\mathbb{R}^{3}}\mathrm{tr}\left(Q^{2}\right)
\end{equation*}
for solutions of \eqref{qteneq} implies that there exists $\gamma'=\gamma'(T, \gamma)>0$ such that the solution $Q(\cdot, t)$ satisfies
\begin{equation*}
Q(\cdot, t)\in K':=\left\{R\in L^{2}(\mathbb{R}^{3})\, : \, \|R\|_{2}\leq \gamma'\right\},
\end{equation*}
for all $0\leq t\leq T$ whenever $Q(\cdot, 0)=Q_{0}\in K\cap H$. With this in mind, we write $\sigma_{T}$ to denote the set of solution trajectories
\begin{equation*}
\sigma_{T}:=\left\{Q\in C([0, T], K') \,:\, Q(x, t) \hspace{2mm}\text{solves} \hspace{2mm} \eqref{qteneq} \hspace{2mm} \text{and} \hspace{2mm} Q(\cdot, 0)\in K\cap H\right\}.
\end{equation*}
Using suitable properties of solutions of the Q-tensor equation \eqref{qteneq}, by the Arzel\`{a}-Ascoli theorem, $\sigma_{T}$ is a compact topological space with respect to the topology induced from $\left(C([0, T]; K'), d_{\infty}\right)$, where
\begin{equation*}
d_{\infty}(Q_{1}, Q_{2}):=\max_{0\leq t\leq T}d_{w}(Q_{1}(t), Q_{2}(t)).
\end{equation*} 
Similarly, the space $\textsf{M}_{0}(\sigma_{T})$ of all Borel probability measures on $\sigma_{T}$ is also a convex, compact Hausdorff topological space.

Let us now begin our construction. For a given $\overline{\mu}\in\textsf{M}_{0}(K)$ supported on $K\cap H$, by the Krein-Milman theorem we know there exists a sequence of families of Dirac delta measures $\{\delta_{\overline{Q}_{j}^{(k)}}\}_{j=1}^{N(k)}$ for $k=1, 2, 3, ...$ satisfying
\begin{equation*}
\int_{K}\varphi\,d\overline{\mu}=\lim_{k\rightarrow\infty}\sum_{j=1}^{N(k)}\vartheta_{j}^{(k)}\int_{K}\varphi \,d\delta_{\overline{Q}_{j}^{(k)}} \quad \text{for all} \quad \varphi\in C(K).
\end{equation*}
Using this approximating sequence of measures, we define a new sequence of probability measures $\{\mu^{(k)}\}_{k=1}^{\infty}\subset \textsf{M}_{0}(\sigma_{T})$ by
\begin{equation*}
\mu^{(k)}:=\sum_{j=1}^{N(k)}\vartheta_{j}^{(k)}\delta_{Q\left(\cdot\, ; \overline{Q}_{j}^{(k)}\right)},
\end{equation*}
where $Q(\cdot \,; \overline{Q}_{j}^{(k)})$ denotes the solution trajectory starting from the initial datum $\overline{Q}_{j}^{(k)}$. This defines a sequence of elements in $\mathsf{M}_{0}(\sigma_{T})$, so by compactness there exists a measure $\mu\in\textsf{M}_{0}(\sigma_{T})$ to which a (relabeled) subsequence of $\{\mu^{(k)}\}_{k=1}^{\infty}$ converges in the weak-$\ast$ topology. We now focus on this limiting measure $\mu$.

For each fixed $s\in [0, T]$, the map
\begin{equation*}
\varphi\mapsto\int_{\sigma_{T}}\varphi(Q(s))\,d\mu
\end{equation*}
is well defined, positive and linear on $C(K)$, so by the Riesz-Kakutani theorem there exists a Borel probability measure $\mu_{s}$ (depending on the choice of $s\in [0, T]$) satisfying
\begin{equation*}
\int_{\sigma_{T}}\varphi(Q(s))\,d\mu(Q)=\int_{K}\varphi(R)\,d\mu_{s}(R) 
\end{equation*}
for all $\varphi\in C(K)$. By direct computation, one may show for the sequence of approximants $\{\mu^{(k)}\}_{k=1}^{\infty}$ that the equality
\begin{equation*}
\int_{\sigma_{T}}\varphi(Q(s))\,d\mu^{(k)}(Q)=\int_{K}\varphi(R)\,d\mu^{(k)}_{s}(R)
\end{equation*}
holds for all $\varphi\in C(K)$, where
\begin{equation*}
\mu^{(k)}_{s}:=\sum_{j=1}^{N(k)}\vartheta_{j}^{(k)}\delta_{S(s)\overline{Q}_{j}^{(k)}}.
\end{equation*}
Using the fact that $\mu^{k}\rightharpoonup \mu$ weakly-star in $\mathsf{M}_{0}(\sigma_{T})$, we find
\begin{align*}
& \quad \int_{K}\varphi(R)\,d\mu_{t}=\int_{\sigma_{T}}\varphi(Q(t))\,d\mu(Q)\\ = & \quad \lim_{k\rightarrow\infty}\int_{\sigma_{T}}\varphi(Q(t))\,d\mu^{(k)}=\lim_{k\rightarrow\infty}\int_{K}\varphi(R)\,d\mu^{(k)}_{t}(R).
\end{align*}
From this we deduce the result
\begin{equation*}\label{measurings}
\mu_{t}(E)=\lim_{k\rightarrow\infty}\mu^{(k)}_{t}(E),
\end{equation*}
for all $\mu_{t}$-measurable subsets $E\subseteq K$. The proof of the theorem is concluded by noting that $\lim_{t\rightarrow 0}\mu_{t}(E)=\overline{\mu}(E)$ for all $E$ by the continuity of the semigroup $\{S(t)\}_{t\geq 0}$, and also that the transport of the measure $\overline{\mu}$ may be extended globally in time as the semigroup $\{S(t)\}_{t\geq 0}$ is defined globally in time.
\end{proof}
Assembling all that has come before, we now approach the proof of our main result, which follows rather swiftly from previous results.
\begin{thm}\label{mainthm}
For any given $\delta>0$, there exists $\eta>0$ depending only on $\delta$ and the parameters $(a, b, c)\in\mathcal{D}$ such that for any Borel probability measure $\mu_{0}$ supported in the dense open set
\begin{equation*}
\mathcal{A}^{\circ}\setminus\mathcal{B}=\left\{R\in L^{\infty}(\mathbb{R}^{3})\,:\, \underset{x\in\mathbb{R}^{3}}{\mathrm{ess}\sup}\left(1+|x|\right)^{8+\delta}|R(x)|< \eta\right\}\setminus\mathcal{B},
\end{equation*}
where $\mathcal{B}$ was defined in remark \ref{bremark}, the associated correlation function \eqref{corrfun} exhibits asymptotic self-similar behaviour:
\begin{equation*}
\left\|c_{\mu_{0}}(r, t)-e^{-\frac{|r|^{2}}{8t}}\right\|_{L^{\infty}(\mathbb{R}^{3}, \, dr)} = \mathcal{O}(t^{-1/2}) \quad \text{as}\quad t\rightarrow\infty.
\end{equation*}
\end{thm}
\begin{proof}
By the previous result \eqref{meas1}, we know that
\begin{equation*}
\sum_{j=1}^{N_{k}}\vartheta_{j}^{k}\delta_{S(t)\overline{Q}_{k, j}}(E)\rightarrow\mu_{t}(E)
\end{equation*}
as $k\rightarrow\infty$ for any measurable subset $E\subset \mathcal{A}^{\circ}\setminus\mathcal{B}$. Since we have that
\begin{equation*}
c_{\mu_{0}}(r, t)=\lim_{k\rightarrow\infty}\frac{\displaystyle\sum_{j=1}^{N(k)}\vartheta_{j}^{(k)}\int_{H}\int_{\mathbb{R}^{3}}\mathrm{tr}\left(Q(x+r)Q(x)\right)\,dxd\delta_{S(t)\overline{Q}_{j}^{(k)}}(Q)}{\displaystyle\sum_{j=1}^{N(k)}\vartheta_{j}^{(k)}\int_{H}\int_{\mathbb{R}^{3}}\mathrm{tr}\left(Q(x)^{2}\right)\,dxd\delta_{S(t)\overline{Q}_{j}^{(k)}}(Q)}.
\end{equation*}
the result then follows by employing a similar calculation as the one used to achieve \eqref{corrindiv}. 
\end{proof}

\section{The Scaling Regime $L(t)=t$: An Observation}
We now present evidence to support the existence of another scaling regime for the correlation function $c_{\mu_{0}}$, namely $L(t)=t$, when $\mu_{0}$ is concentrated on initial profiles whose $L^{2}(\mathbb{R}^{3})$-norm is \emph{not} restricted in magnitude in the manner of theorem \ref{thm1}. 

Suppose that we offer initial data $Q_{0}\in H$ of the form $Q_{0}=\mathrm{diag}(\lambda_{0}, \lambda_{0}, -2\lambda_{0})$ for equation \eqref{qteneq}. By uniqueness (proposition \ref{globexpro}) we know the solution $Q$ remains of diagonal form $Q=\mathrm{diag}(\lambda, \lambda, -2\lambda)$, where $\lambda$ satisfies the \emph{scalar} nonlinear heat equation
\begin{equation}\label{lambdaeq}
\frac{\partial\lambda}{\partial t}=\Delta\lambda - a\,\lambda- b\,\lambda^{2}-6c\,\lambda^{3}.
\end{equation}
For any $R>0$ we denote by $\lambda_{0, R}:\mathbb{R}^{3}\rightarrow \mathbb{R}$ the map
\begin{displaymath}
\lambda_{0, R}(x):=\left\{
\begin{array}{ll}
\displaystyle \lambda^{\ast}  & \quad \text{if} \quad |x|< R, \vspace{2mm} \\ \displaystyle 0  & \quad \text{otherwise,} 
\end{array}
\right.
\end{displaymath}  
where $\lambda^{\ast}$ is the global minimiser of the bulk potential associated with \eqref{lambdaeq}, namely
\begin{equation*}
\lambda^{\ast}=\min_{\lambda\in\mathbb{R}}\left(\frac{a}{2}\lambda^{2}+\frac{b}{3}\lambda^{3}+\frac{3c}{2}\lambda^{4}\right).
\end{equation*}
In order to comment on the asymptotic behaviour of the correlation function concentrated on such solutions, we require the following result (which we state without proof) that is based on the observations of \textsc{Aronson and Weinberger} \cite{aronson1} and \textsc{Pol\'{a}\u{c}ik} \cite{polacik1} that equation \eqref{lambdaeq} supports \emph{pulse-like} solutions.
\begin{prop}
There exist $R_{0}>0$ and $\overline{c}>0$ such that solutions of \eqref{lambdaeq} subject to initial data $\lambda_{0, R}$ for $R\geq R_{0}$ satisfy
\begin{equation}\label{asyml}
\lim_{t\rightarrow \infty}\|\lambda(\cdot, t)-\lambda^{\ast}\|_{L^{\infty}(B(0, \overline{c}t))}=0,
\end{equation}
where $B(0, \overline{c}t)\subset\mathbb{R}^{3}$ denotes the ball of radius $\overline{c}t$ and centre $0$. Furthermore, for any $t_{0}>0$ there exists $\vartheta=\vartheta(t_{0})$ such that the solution $\lambda$ satisfies the bounds
\begin{equation}\label{asymm}
0\leq \lambda(x, t)\leq Ce^{-\sigma(|x|-\overline{c}t+\vartheta)} \quad \text{on} \quad \mathbb{R}^{3}\times [t_{0}, \infty),
\end{equation}
for some positive constants $C$ and $\sigma$.
\end{prop}
Utilising the results \eqref{asyml} and \eqref{asymm} above, one may verify that the correlation function
\begin{equation*}
c_{\delta_{Q_{0}}}(r, t)=\frac{\displaystyle\int_{\mathbb{R}^{3}}\lambda(x+r, t)\lambda(x, t)\,dx}{\displaystyle\int_{\mathbb{R}^{3}}\lambda(x, t)\lambda(x, t)\,dx}
\end{equation*} 
concentrated on solutions with initial data $Q_{0}=\mathrm{diag}(\lambda_{0, R}, \lambda_{0, R}, -2\lambda_{0, R})$ for $R\geq R_{0}$ satisfies
\begin{equation*}
\lim_{t\rightarrow\infty}\left\|c_{\delta_{Q_{0}}}(r, t)-P\left(\frac{|r|}{t}\right)\right\|_{L^{\infty}(\mathbb{R}^{3}, \,dr)}=0,
\end{equation*}
where $P$ is a cubic polynomial of the shape
\begin{equation*}
P(z)=\frac{1}{13\overline{c}^{3}}\left(4\overline{c}+z\right)\left(2\overline{c}-z\right)^{2}.
\end{equation*}
Thus, for this one-parameter family of initial data, whose $H$-norm may be made arbitrarily large, we obtain \emph{both} a different universal scaling function (a polynomial, as opposed to the exponential map) and a different form for the coarsening length scale $L(t)$. Some effort would be required to establish the stability of such profiles, leading to a result for the correlation function $c_{\mu_{0}}$, where $\mu_{0}$ is a non-atomic measure supported on an appropriate set of `large' initial data. Finally, it would be of some interest to establish asymptotic self-similarity of $c_{\mu_{0}}$ when $\mu_{0}$ is supported on the class of `large' initial data which look like the union of nematic `islands' previously alluded to in the introduction of this paper. We hope to investigate this problem in future work.
\section{Closing Remarks}
In this paper, we have been able to obtain asymptotic information for $c_{\mu_{0}}(\cdot, t)$ in $L^{\infty}(\mathbb{R}^{3})$ for $\mu_{0}$ probing a rather restricted region of phase space $H$. Ideally, one would like to be able to decompose phase space $L^{2}(\mathbb{R}^{3})\cap L^{6}(\mathbb{R}^{3})=\cup_{j\in J}H_{j}$ in such a manner that given the knowledge $\mathrm{supp}\,\mu_{0}\subseteq H_{i}\subset H$, one could find non-trivial $\Gamma_{i}:\mathbb{R}^{3}\rightarrow\mathbb{R}$ and $L_{i}:(0, \infty)\rightarrow (0, \infty)$ such that
\begin{equation*}
\lim_{t\rightarrow\infty}\left\|c_{\mu_{0}}(r, t)-\Gamma_{i}\left(\frac{r}{L_{i}(t)}\right)\right\|_{L^{\infty}(\mathbb{R}^{3}, dr)}=0.
\end{equation*}
It is in this sense we hope to construct a {\em phase portrait} for $c_{\mu_{0}}$ on $H$ in future work.

Furthermore, our main result is a `small initial data' result which one would expect to hold for a wide class of reaction-diffusion equations, and in particular for scalar equations. It is not clear whether or not the high-dimensional target space of the system \eqref{gradientflowings} supports behaviour of $c_{\mu_{0}}$ distinct from that of correlation functions defined for solutions of scalar equations.

\bibliography{sshreferences}

\vspace{10mm}
\begin{tabular}{l}
{\Large Eduard Kirr}  \\
Department of Mathematics,\\
The University of Illinois at Urbana-Champaign,\\
Illinois, USA.\\
\href{mailto: ekirr@math.uiuc.edu}{\nolinkurl{ekirr@math.uiuc.edu}}\vspace{5mm}\\

{\Large Mark Wilkinson (\Letter)}\\
D\'{e}partement de Math\'{e}matiques et Applications, \\
\'{E}cole Normale Sup\'{e}rieure,\\
Paris, France.\\
\href{mailto:mark.wilkinson@ens.fr}{\nolinkurl{mark.wilkinson@ens.fr}} \vspace{5mm}\\

{\Large Arghir Zarnescu}\\
Department of Mathematics,\\
University of Sussex,\\
Brighton, UK.\\
\href{mailto: a.zarnescu@sussex.ac.uk}{\nolinkurl{a.zarnescu@sussex.ac.uk}}
\end{tabular}
\end{document}